\documentclass[sn-mathphys]{sn-jnl}

\usepackage{latexsym}
\usepackage{amsmath}
\usepackage{mathtools}
\usepackage{amssymb}
\usepackage{bbm}
\usepackage{ulem}
\usepackage{color}
\usepackage[shortlabels]{enumitem}
\usepackage{hyperref}

\DeclareFontFamily{U}{mathx}{}
\DeclareFontShape{U}{mathx}{m}{n}{ <-> mathx10 }{}
\DeclareSymbolFont{mathx}{U}{mathx}{m}{n}
\DeclareFontSubstitution{U}{mathx}{m}{n}

\DeclareMathAccent{\widecheck}{0}{mathx}{"71}

\newcommand{\Mp}{M_p(\psi,\phi;\psi^*,\phi^*)}

\newcommand{\Calphamax}{\mathcal{G}}
\newcommand{\Calpha}{\mathcal{G}_\alpha}
\newcommand{\gCZ}{\mathcal{Z}}
\newcommand{\CZ}{\mathcal{C}}
\newcommand{\N}{\mathcal{N}}
\newcommand{\Np}{\mathcal{N}_p}
\newcommand{\K}{\mathcal{K}}
\newcommand{\atom}{h}

\jyear{2021}%

\theoremstyle{thmstyleone}%
\newtheorem{theorem}{Theorem}
\newtheorem{proposition}{Proposition}
\newtheorem{lemma}{Lemma}

\theoremstyle{thmstyletwo}%
\newtheorem{remark}{Remark}%
\newtheorem*{nonumremark}{Remark}%

\theoremstyle{thmstylethree}%
\newtheorem{definition}{Definition}%

\begin{document}
	
	\title[Wavelet Series Expansion in Hardy Spaces with Approximate Duals]{Wavelet Series Expansion in Hardy Spaces with Approximate Duals}
	
	
	\author*[1]{\fnm{Youngmi} \sur{Hur}}\email{yhur@yonsei.ac.kr}
	
	\author[2]{\fnm{Hyojae} \sur{Lim}}\email{hyo5064@yonsei.ac.kr}
	
	
	\affil*[1]{\orgdiv{Department of Mathematics}, \orgname{Yonsei University},\orgaddress{ \state{Seoul}, \country{Korea}}}
	
	\affil[2]{\orgdiv{School of Mathematics and Computing (Mathematics)}, \orgname{Yonsei University}, \orgaddress{\state{Seoul}, \country{Korea}}}
	

	\abstract{In this paper, we provide sufficient conditions for the functions \(\psi\) and \(\phi\) to be the approximate duals in the Hardy space \(H^p(\mathbb{R})\) for all \(0<p\le 1\). 
	Based on these conditions, we obtain the wavelet series expansion in the Hardy space \(H^p(\mathbb{R})\) with the approximate duals. The important properties of our approach include the following: (i) our results work for any \(0<p \leq 1\); (ii) we do not assume that the functions \(\psi\) and \(\phi\) are exact duals; (iii) we provide a tractable bound for the operator norm of the associated wavelet frame operator so that it is possible to check the suitability of the functions \(\psi\) and \(\phi\).
}


	
	
	\maketitle
	\bmhead{Acknowledgments}
	This work was supported in part by the National Research Foundation of Korea (NRF) [Grant Numbers 2015R1A5A1009350 and 2021R1A2C1007598].
	
\section{Introduction}
\label{sec:intro}
\subsection{Preliminaries}

Let $W_A(\psi)$ be the wavelet system with dilation factor $A>1$ generated by $\psi\in L^2(\mathbb{R})$, i.e. $W_A(\psi):=\{\psi_{jk}:=A^{j/2} \psi(A^j\cdot-k):j,k\in\mathbb{Z}\}$. We recall that $W_A(\psi)$ is a wavelet frame in $L^2(\mathbb{R})$ if, for every $f\in L^2(\mathbb{R})$,
$$	
f = \sum\limits_{j,k \in \mathbb{Z}} \left< f, \phi_{jk} \right> 
	\psi_{jk}
$$
with its dual wavelet frame $W_A(\phi)$. 
In this case, the two generating functions \(\psi\) and \(\phi\) are referred to as {\it (exact) duals} in \(L^2(\mathbb{R})\) \cite{[BuiLaug]FSF,[ChristensenLaugesen]ApproximateDual}.
	
Wavelet frames have been used in literature to characterize function spaces such as Triebel-Lizorkin spaces and Besov spaces\footnote{These characterizations are studied mostly under the dyadic dilation setting with $A=2$, although many of them can be easily extended to a more general case with dilation $A>1$.}. In particular, for the Hardy space $H^p(\mathbb{R})$, $0<p\le 1$, if the functions \(\psi\) and \(\phi\) are exact duals in \(L^2(\mathbb{R})\), and they satisfy smoothness and vanishing moment conditions depending on $p$, then any function $f\in H^p(\mathbb{R})$ has a wavelet series expansion of the form \cite{[FJ]DB,[FJ]DD,[FJW]LPT,[Kyr]DSFS}
	\begin{equation}
	f = \sum\limits_{j,k\in\mathbb{Z}} \left<f,\phi_{jk}\right> \psi_{jk}.
	\label{eq:classic_WSE}
	\end{equation}

Approximate duals have been studied as a way to generalize the exact duals in \(L^2(\mathbb{R})\). Approximate duals in \cite{[Christensen]IFR,[ChristensenLaugesen]ApproximateDual} are the functions \(\psi\) and \(\phi\) whose associated wavelet frame operator \(U\) defined as
	\[U (g) := \sum\limits_{j,k\in\mathbb{Z}} \left<g,\phi_{jk}\right> \psi_{jk}\]
{\it approximates} the identity operator in \(L^2(\mathbb{R})\) in the sense that \(\|U-Id\|_{L^2\to L^2}<1\). 

More recently, Bui and Laugesen study approximate duals in \(H^1(\mathbb{R})\) \cite{[BuiLaug]WFBonLH}. They identify the conditions for the two generating functions \(\psi\) and \(\phi\) to satisfy in order to be approximate duals in \(H^1(\mathbb{R})\) in the sense that \(\left\| U - Id \right\|_{H^1\to H^1}<1\). It is also shown that these approximate duals give a wavelet series expansion of the form
\begin{equation}
f = \sum\limits_{j,k\in\mathbb{Z}} \left< U^{-1} f , \phi_{jk} \right> \psi_{jk}
\label{eq:BL_WSE}
\end{equation}
for any function $f\in H^1(\mathbb{R})$. 
Here, \(U^{-1}f\) makes sense since, in this case, \(U\) is bounded as well.

This wavelet series expansion can be understood as a generalization of the aforementioned classical wavelet series expansion in \eqref{eq:classic_WSE} for \(H^1(\mathbb{R})\). In fact, if $U$ is the identity operator in \(H^1(\mathbb{R})\), then \(\psi\) and \(\phi\) are exact duals in \(H^1(\mathbb{R})\) and the series expansion in \eqref{eq:BL_WSE} reduces to the series expansion in \eqref{eq:classic_WSE}. 

Of course, when we already have the exact duals \(\psi\) and \(\phi\) in \(L^2(\mathbb{R})\), we can check to see if they satisfy additional conditions to get the series expansion in~\eqref{eq:classic_WSE}. However, for this process to go smoothly, we have to know both \(\psi\) and \(\phi\) concretely enough to check these conditions, which may not always be possible. 

For example, it is well known that the Mexican hat function \(\psi\) has its dual \(\phi\) so that \(W_2(\psi)\) and \(W_2(\phi)\) are dual wavelet frames in \(L^2(\mathbb{R})\) \cite{[Daubechies]TenLecture}. Although the dual function \(\phi\) of \(\psi\) is not unique, to the best of our knowledge, the dual function \(\phi\) has never given explicitly enough to allow us to check additional conditions for \(H^1(\mathbb{R})\) characterization, resulting the series expansion in \eqref{eq:classic_WSE}.

On the other hand, by showing that the Mexican hat function \(\psi\) and some carefully chosen function \(\phi\) are approximate duals in \(H^1(\mathbb{R})\), it is proved in \cite{[BuiLaug]WFBonLH} that the Mexican hat function \(\psi\) provides 
\[f = \sum\limits_{j,k\in\mathbb{Z}} \left< U^{-1} f , \phi_{jk} \right> \psi_{jk},\quad \forall \, f\in H^1(\mathbb{R}).\]

In this paper, we extend the above result on \(H^1(\mathbb{R})\) by Bui and Laugesen to cover the full range of the Hardy spaces, i.e. \(H^p(\mathbb{R})\) for \(0<p \leq 1\). 

\medskip

Let $L^2:=L^2(\mathbb{R})$ and $H^p:=H^p(\mathbb{R})$. 	The (mixed) wavelet frame operator \(U: L^2 \to L^2\) is studied in recent papers \cite{[BuiLaug]FSF,[BuiLaug]WLPMHC,[BuiLaug]WFBonLH}, and we recall its definition below.
		\begin{definition}
		Let \(\psi, \phi\in L^2\). Let \(A>1\) be the dilation factor. The \textit{wavelet frame operator} \(U:=U_{\psi,\phi}\) associated with a \textit{synthesizer} \(\psi\) and an \textit{analyzer} \(\phi\) is defined as
		\[U_{\psi,\phi} (f) := \sum\limits_{j,k\in\mathbb{Z}} \left<f,\phi_{jk}\right> \psi_{jk}\]
		where \(\psi_{jk} = A^{j/2} \psi(A^j\cdot-k)\) and \(\phi_{jk} = A^{j/2} \phi(A^j\cdot-k)\), as before.	
		\label{def:wavelet operator U}
	\end{definition}

We recall a result from \cite{[BuiLaug]WFBonLH} (see also \cite[Proposition 6 and 7]{[BuiLaug]FSF}) providing sufficient conditions on \(\psi\) and \(\phi\) for the operator \(U\) to be bounded on \(L^2\). Furthermore, the operator can be written in an integral form with the wavelet frame kernel \(K_{\psi,\phi}\) defined below. Throughout the paper, \(W^{k,p}\) is used to denote the Sobolev space, for \(k\in \mathbb{N}\) and \(1 \leq p <\infty\).
	
	\begin{theorem}[\cite{[BuiLaug]WFBonLH}]
		Assume \(\psi, \phi \in L^2\) satisfy
		\begin{equation}
			\lvert\widehat{\psi}(\xi)\rvert, \lvert\widehat{\phi}(\xi)\rvert \lesssim \begin{cases}
				\lvert\xi\rvert^\epsilon,  & \lvert\xi\rvert\leq1,\\
				\lvert\xi\rvert^{-\epsilon-\frac{1}{2}}, & \lvert\xi\rvert\geq1,
			\end{cases}
			\quad \hbox{for some } \epsilon>0.
			\label{eq:BLdecay}
		\end{equation}
		Assume \(\widehat{\psi}, \widehat{\phi} \in W^{2,1}\). Then the wavelet frame operator \(U=U_{\psi,\phi}\) is bounded and linear on \(L^2\), and if \(f\in L^2\) has compact support and \(x \notin \text{supp}f\), then
		\[(U_{\psi,\phi}\ f)(x) = \int_{\mathbb{R}} K_{\psi,\phi}(x,y) f(y) dy\]
		where the wavelet frame kernel \(K_{\psi,\phi}\) is defined as
		\begin{equation}
			K_{\psi,\phi}(x,y) := \sum\limits_{j \in\mathbb{Z}} \sum\limits_{k \in\mathbb{Z}} \psi_{jk}(x) \overline{\phi_{jk}(y)}.
			\label{WaveletFrameKernel}
		\end{equation}
		\label{Result: L2boundedness}
	\end{theorem}

In \cite{[BuiLaug]WFBonLH}, by imposing sufficient decay conditions and smoothness conditions with respect to the frequency domain of the synthesizer \(\psi\) and the analyzer~\(\phi\), it is shown that the wavelet frame operator \(U\) is the \(L^2\)-based Calder\'{o}n-Zygmund operator\footnote{Since we focus on the Calder\'{o}n-Zygmund operator associated with \(p\) (c.f., Defintion~\ref{def: Calderon-Zygmund}) throughout the paper, we call the standard Calder\'{o}n-Zygmund operator, defined in \cite{[Grafakos]CFA,[HW]FCW}, \(L^2\)-based.} (see, for example, \cite{[Grafakos]CFA,[HW]FCW} for definition), and has a bounded extension from \(L^2\) to  \(H^1\), \(L^p\) (\(1<p<\infty\)), and the BMO space. Also shown is the invertibility of \(U\) by assuming a sufficient computable condition on \(\psi\) and \(\phi\) to be approximate duals. This invertibility result is in turn used to show that every element in these function spaces has a wavelet series expansion via \(W_A(\psi)\). These results allow the Mexican hat function to be an explicit example for \(\psi\) on these function spaces, including \(H^1\).
	
In this paper, we extend the above result of \(H^1\) to the Hardy space \(H^p\) for all \(0<p \leq 1\). Specifically, we obtain the boundedness and the invertibility of the wavelet frame operator \(U\) on \(H^p\), \(0<p \leq 1\), by using the appropriately generalized Calder\'{o}n-Zygmund operator for the full range of Hardy spaces \(H^p\). Subsequently, we show the wavelet series expansion with the wavelet system \(W_A(\psi)\) on the Hardy space, by applying the generalized result of Frazier and Jawerth, and by showing the equivalence between the Hardy space \(H^p\) and the (non-dyadic) Triebel-Lizorkin space \(\dot{F}^{02}_p\) (see Definition \ref{def:TLspace,sequenceSpace}\ref{def:TLspace}). We also show that the Mexican hat function can be used as an explicit example as the synthesizer \(\psi\) on \(H^p\) for all \(1/2<p \leq 1\).

To obtain our main results, we mostly follow the approach of \cite{[BuiLaug]WFBonLH}, but also employ tools developed in \cite{[BuiLaug]FSF,[BuiLaug]WLPMHC,[BuiLaug]WFBonLH} as well as the classical theory in \cite{[Bownik]AniHardy,[GarciaFrancia]WNI,[Grafakos]MFA,[Meyer]CZM,[Stein]HA} to handle a Calder\'{o}n-Zygmund operator and the \(\varphi\)-transform with proper modifications as needed. 

\medskip 	

We use the following definition of a Calder\'{o}n-Zygmund operator associated with \(0<p \leq 1\) (c.f., \cite{[Meyer]WaO}) which is properly generalized to apply to the Hardy space \(H^p\). As in many results and definitions in our paper, the conditions in this definition depend on \(p\), through the floor function
	\[\Np:=\lfloor 1/p - 1 \rfloor.\]
Note that \(\Np =0\) if \(p=1\), or more generally, if \(1/2 < p \leq 1\). 


\begin{definition}
	\label{def: Calderon-Zygmund}
	
	A \textit{Calder\'{o}n-Zygmund operator} \(Z\) (associated with \(p\)) is a bounded linear operator on \(L^2\) such that for compactly supported \(f\in L^2\),
	\[
	Zf(x) = \int_{\mathbb{R}} \K(x,y) f(y) dy \quad \text{for} \quad x\notin \text{supp} f
	\]
	where the kernel \(\K: (\mathbb{R} \times \mathbb{R}) \setminus \{(x,x): x\in\mathbb{R}\} \to \mathbb{C}\) is a measurable function satisfying the following conditions:
	\begin{enumerate}[(a)]
		\item There is a constant \(\gCZ>0\) such that 
		\begin{equation}
			\lvert \partial^\alpha_{y} \K(x,y)\rvert \leq \frac{\gCZ}{\lvert x-y \rvert^{\alpha+1}}, \quad \forall \, 0\leq \alpha \leq \Np +1.
			\label{CZCond2}
		\end{equation}
		\item For any compactly supported \(f\in L^2\) with \(\int f(x) x^\beta dx = 0\) for all \(0\leq \beta \leq \Np\),
		\begin{equation}
			\int Zf(x) x^\alpha dx = 0, \quad \forall \,0 \leq \alpha \leq \Np . 
			\label{CZCond3} 
		\end{equation}
	\end{enumerate}
		\label{def:CZO}
\end{definition}

	\begin{remark} 
		The \(L^2\)-based Calder\'{o}n-Zygmund operator studied in the literature (c.f.,~\cite{[HW]FCW}) does not require any vanishing moment condition such as (\ref{CZCond3}), but requires the smoothness conditions:
		\begin{equation*}
			\lvert \partial^\alpha_{y} \K(x,y)\rvert \leq \frac{\gCZ}{\lvert x-y \rvert^{\alpha+1}}, \quad \alpha= 0,1.
		\end{equation*}
		In some literature (e.g., \cite{[Grafakos]CFA}), these smoothness conditions are replaced by slightly weaker smoothness conditions involving the Lipschitz continuity.
		\qed
	\end{remark}
	\begin{remark} 
		When \(p=1\) (and thus \(\Np=0\)), the conditions in Definition~\ref{def: Calderon-Zygmund} are similar to those in the Calder\'{o}n-Zygmund operator introduced in \cite{[BuiLaug]WFBonLH} for the Hardy space \(H^1\), where smoothness conditions with Lipschitz continuity are used. \qed
	\end{remark}

\subsection{Main Theorems}

In this subsection, we state our main theorems. Our first main theorem gives sufficient conditions on \(\psi, \phi \in L^2\) for the wavelet frame operator \(U=U_{\psi,\phi}\) to be a Calder\'{o}n-Zygmund operator, and have a bounded extension to \(H^p\).
	
	\begin{theorem}[Boundedness of Wavelet Frame Operator on $H^p$]
		Let \(0<p\leq 1\) be fixed and let \(\N:=\Np\). Assume that \(\psi,\phi \in L^2\) satisfy 
		\begin{equation*}
			\lvert \widehat{\psi}^{(\N+1)}(\xi)\rvert \lesssim \begin{cases}
				\lvert \xi \rvert^\epsilon, & \lvert \xi \rvert \leq 1,\\
				\lvert \xi \rvert^{-\epsilon-\N-3/2}, & \lvert \xi \rvert \geq 1,
			\end{cases} \quad \displaystyle \lvert \widehat{\phi}^{(\N+1)}(\xi) \rvert \lesssim \begin{cases}
				\lvert \xi \rvert^\epsilon, & \lvert \xi\rvert \leq 1,\\
				\lvert\xi\rvert^{-\epsilon-2\N-5/2}, & \lvert \xi \rvert \geq 1
			\end{cases}
		\end{equation*}
		for some \(\epsilon>0\) and the following additional conditions:
		\begin{enumerate}[(a)]
			\item \(\widehat{\psi}\in W^{\N +3,1}\), and \(\xi^\alpha \widehat{\phi} \in W^{\N +3,1}\) for \(\alpha=0,1,\cdots, \N +1\);
			\item \(\widehat{\psi}\in W^{\N +3,2},\widehat{\phi}\in W^{\N+1,2}\), and \(\xi^\alpha \widehat{\phi} \in W^{\alpha+2,2}\) for \(\alpha=1,\cdots, \N +1\);
			\item \(\displaystyle \int \psi(x)\, x^\beta dx  = \int \phi(x) x^\beta dx= 0\), \(\forall \, 0 \leq \beta \leq \N.\)
		\end{enumerate}
		Then the wavelet frame operator \(U=U_{\psi,\phi}\) is a Calder\'{o}n-Zygmund operator, and has a bounded extension to \(H^p\).
		\label{Theorem:Bdd of U}
	\end{theorem}
	
	
	
	Compared with Theorem~\ref{Result: L2boundedness}, the above theorem has the moment conditions for the functions \(\psi\) and \(\phi\), which are typical for the Hardy spaces \(H^p\), \(0<p \leq 1\). In fact, the above conditions of Theorem~\ref{Theorem:Bdd of U} are all dedicated to show that the wavelet frame operator \(U\) is a Calder\'{o}n-Zygmund operator associated with~\(p\) (c.f., Theorem~\ref{thm:CZ kernel smoothness} and \ref{thm:CZ kernel vanishing}). 
	The bounded extension to \(H^p\) of the frame operator can be seen by simply showing that the Calder\'{o}n-Zygmund operator has a bounded extension.

	

	In the next theorem, we use the value \(\Mp\) which will be defined and explained later in Section~\ref{sec: Proof of Theorem2,3}. This theorem asserts that if this value is less than \(1\), then the wavelet frame operator~\(U\) is bijective on \(H^p\). 

	\begin{theorem}[Bijectivity of Wavelet Frame Operator on $H^p$]
		Let \(0<p \leq 1\) be fixed. Assume that \(\psi, \phi\) satisfy all the conditions in Theorem \ref{Theorem:Bdd of U}. Assume also that \(\psi^*, \phi^*\) satisfy all the conditions in Theorem \ref{Theorem:Bdd of U} and they are exact duals in \(L^2\).
		If \(\Mp<1\), then the wavelet frame operator \(U\) defined on \(L^2\) has a bounded linear bijective extension to \(H^p\).
		\label{Theorem:Invertibility of U}
	\end{theorem}

	We will show that the wavelet frame operator \(U = U_{\psi,\phi}\) satisfies \(\left\|U_{\psi,\phi}-Id\right\|_{H^p\to H^p}<1\) from \(\Mp<1\) together with other assumptions in Theorem~\ref{Theorem:Invertibility of U}. That is, the functions \(\psi\) and \(\phi\) are  approximate duals in \(H^p\). These approximate duals \(\psi\) and \(\phi\) in \(H^p\) are used in the next theorem to provide a wavelet series expansion for any function \(f\in H^p\) with an additional regularity condition.
	Below, \(\dot{f}^{02}_p\) denotes the (non-dyadic) Triebel-Lizorkin sequence space defined in Definition \ref{def:TLspace,sequenceSpace}\ref{def:TLSequenceSpace}.

	\begin{theorem}[Wavelet Series Expansion in $H^p$]
		Let \(0<p \leq 1\) be fixed. Assume that \(\psi, \phi, \psi^*, \phi^*\) satisfy all the conditions in Theorem~\ref{Theorem:Invertibility of U} including \(\Mp<1\). Further assume that \(\xi\widehat{\psi} \in W^{\N+2,1}\). Then for any \(f\in H^p\), there exist coefficients \(\{c_{jk}\}_{j,k\in\mathbb{Z}}\in \dot{f}^{02}_p\) such that 
		\[f = \sum\limits_{j,k\in\mathbb{Z}} c_{jk} \psi_{jk}.\]
		Moreover, the coefficient \(c_{jk}\) can be taken as \(\left< U^{-1} f , \phi_{jk} \right>\), where \(U^{-1}\) is the inverse of the wavelet frame operator \(U\) on \(H^p\).
		\label{Corollary}
	\end{theorem}

	Let us set some notations used throughout the paper. We use \(\mathbb{N}_0\) to denote the set of non-negative integers. That is, \(\mathbb{N}_0 := \mathbb{N} \cup \{0\}\). We use the notation \(f \lesssim g\) to denote \(f \leq C g\) for some constant \(C>0\). 
	In addition, for \(\alpha\in\mathbb{N}\), we denote the \(\alpha\)-th derivative of \(f\) by \(f^{(\alpha)}\), and the \(\alpha\)-th partial derivative of \(K(x,y)\) with respect to \(x\) by \(\partial^{\alpha}_x K(x,y)\). For \(\alpha=1\), we use both \(f^{(1)}\) and \(f'\), and both \(\partial^1_x K(x,y)\) and \(\partial_x K(x,y)\). By \(f^{(0)}\) and \(\partial_x^0 K(x,y)\), we refer to \(f\) and \(K(x,y)\), respectively, so that the expression``\(0\)-th derivative'' can be used. For \(f,g\in L^2\), \( \left< f, g \right> := \int_\mathbb{R} f \overline{g} dx\). The Fourier transform is normalized as, for \(f\in L^1 \cap L^2\),
	\[\hat{f}(\xi) = \int_\mathbb{R} f(x) e^{-2\pi ix \xi}dx,\]
	and the inverse Fourier transform of \(f\) is denoted by \(\widecheck{f}\). 
	
	\medskip
	
	The rest of the paper is organized as follows. In Section \ref{sec:WFO is CZO}, we find the conditions on \(\psi,\phi \in L^2\) with respect to the frequency domain for the wavelet frame operator \(U = U_{\psi,\phi}\) to be a Calder\'{o}n-Zygmund operator with a computable constant. In Section \ref{sec:Bdd of U on Hardy}, we define the Hardy space \(H^p\) using atomic decomposition, and show that a Calder\'{o}n-Zygmund operator can be extended to the Hardy space \(H^p\) for each \(0<p \leq 1\). An explicit bound of the operator norm is found in this section. In Section~\ref{sec: Proof of Theorem2,3}, we prove Theorem~\ref{Theorem:Bdd of U} and \ref{Theorem:Invertibility of U}. In Section \ref{sec:proof of corollary}, we prove Theorem~\ref{Corollary}. In Section~\ref{sec:Example}, we give an illustration of our main theorems using the Mexican hat function as an example for the case when \(1/2<p \leq 1\). Finally, some more technical lemmas, example, and proofs are placed in Appendix~\ref{appx : lemmas needed in Sec2}, \ref{appx : example of Lemma9} and \ref{appx: Necessities for Thm6}.
	
	\section{Wavelet Frame Operator as a Calder\'{o}n-Zygmund Operator}
	\label{sec:WFO is CZO}
	
	\subsection{Main Results}
	\label{subsec:Main Result on CZO}
	
	In the next two theorems, we provide sufficient conditions on \(\psi\) and \(\phi\) for the wavelet frame kernel \(K_{\psi,\phi}\) in (\ref{WaveletFrameKernel}) to satisfy the smoothness condition (\ref{CZCond2}) and the vanishing moment condition (\ref{CZCond3}) of a Calder\'{o}n-Zygmund operator. This means that, by Theorem~\ref{Result: L2boundedness}, the wavelet frame operator \(U\) is a Calder\'{o}n-Zygmund operator defined with the kernel \(K_{\psi,\phi}\). The case when \(p=1\) has already been studied in Proposition 4.5 and Theorem 4.6 of \cite{[BuiLaug]WFBonLH}, and the sufficient conditions on \(\psi\) and \(\phi\) for \(p=1\) correspond to our conditions below with \(\Np=0\). Proofs of the following theorems are presented in Section~\ref{subsec: Necessities for Thm4} and \ref{subsec: Necessities for Thm5}. 
	
	Here and below, we use the two quantities defined as, for \(\alpha\in \mathbb{N}_0\),
		\begin{equation}
			\label{eq:sigma_alpha}
			\sigma_\alpha(\psi,\phi) := (2\pi)^{\alpha} \sum\limits_{l\in\mathbb{Z}} \left\| \xi^\alpha\overline{\widehat{\phi}(\cdot)} \widehat{\psi}\left(\cdot+l\right)\right\|_{L^1},
		\end{equation}
		\begin{equation}
			\label{eq:tau_alpha}
			\tau_\alpha(\psi,\phi) := \frac{1}{4\pi^2} \sum\limits_{l\in\mathbb{Z}} \left\| \left( \xi^\alpha \overline{\widehat{\phi}(\cdot)} \widehat{\psi}\left(\cdot+l\right) \right)^{ (\alpha+2) } \right\|_{L^1}.
		\end{equation}
	
	\begin{theorem}
		Let \(0<p \leq 1\) be fixed and let \(\N:=\Np\). Also, let \(A>1\) be a dilation factor. Assume that \(\psi, \phi \in L^2\) satisfy the following conditions: 
		\begin{enumerate}[(a)]
			\item \(\widehat{\psi}\in W^{\N+3,1}\) and \(\xi^\alpha \widehat{\phi} \in W^{\N +3,1}\) for \(\alpha=0,1, \cdots, \N+1\),
			
			\item \(\widehat{\psi}\in W^{ \N +3,2}\) and \(\xi^\alpha \widehat{\phi} \in W^{\alpha+2,2}\) for \(\alpha=0,1, \cdots, \N+1\).
		\end{enumerate}
		We also assume \(\sigma_\alpha(\psi,\phi)<\infty\) and \(\tau_\alpha(\psi,\phi)<\infty\) for all \(0 \leq \alpha \leq \N +1\). Then for any \(x,y\in \mathbb{R}\) such that \(x \neq y\), 
		every \(\alpha\)-th partial derivative of wavelet frame kernel \(K_{\psi,\phi}\) defined in (\ref{WaveletFrameKernel}) is bounded as follows: 
		\[\lvert \partial^\alpha_y K_{\psi,\phi}(x,y) \rvert \leq \frac{1}{\lvert x-y \rvert^{\alpha+1} }\left(\max\limits_{0 \leq \alpha \leq  \N+1} \CZ_{\alpha}(\psi,\phi)\right), \quad \forall\, 0 \leq \alpha \leq \N+1,\]		
		where 
		\begin{equation}
			\CZ_{\alpha}(\psi,\phi) := \kappa_\alpha(A) \sigma_\alpha(\psi,\phi)^{1/(\alpha+2)} \tau_\alpha(\psi,\phi)^{(\alpha+1)/(\alpha+2)}
			\label{eq:CZ_2,alpha}
		\end{equation}
		with a constant \(\kappa_\alpha(A) = A(2A^\alpha + \sum_{k=0}^{\alpha-1} A^k)/(A^{\alpha+1} -1)\) (here, \(\sum_{k=0}^{-1} A^k:=0\)).
		\label{thm:CZ kernel smoothness}
	\end{theorem}

	\begin{theorem}
		Let \(0<p \leq 1\) be fixed and let \(\N :=\Np\). Assume that \(\psi,\phi\in L^2\) satisfy 
		\begin{equation*}
			\lvert \widehat{\psi}^{(\N+1)}(\xi) \rvert \lesssim \begin{cases}\lvert \xi \rvert^\epsilon, & \lvert \xi \rvert \leq 1,\\ \lvert \xi \rvert^{-\epsilon-\N-3/2}, & \lvert \xi \rvert \geq 1,\end{cases} \quad \lvert \widehat{\phi}^{(\N+1)}(\xi) \rvert \lesssim \begin{cases}\lvert \xi \rvert ^\epsilon, & \lvert \xi \rvert \leq 1,\\ \lvert \xi \rvert^{-\epsilon-2\N-5/2}, & \lvert \xi\rvert \geq 1\end{cases}
		\end{equation*}
		for some \(\epsilon>0\), and the following additional conditions:
		\begin{enumerate}[(a)]
			\item \(\widehat{\psi}, \widehat{\phi} \in W^{\N+1,2}\),
			\item \( \int \psi(x) x^\alpha dx=  \int \phi(x) x^\alpha dx= 0\), \(\forall \,0 \leq \alpha \leq \N\).
		\end{enumerate}
		If \(f\in L^2\) has compact support and \(\int f(x) x^\beta dx = 0 \) for all \(0 \leq \beta \leq \N\), then the wavelet frame operator \(U\) satisfies the following vanishing moment condition:
		\begin{equation*}
			\int  Uf(x) x^\alpha dx= 0, \quad \forall \,0 \leq \alpha \leq \N.
		\end{equation*}
		\label{thm:CZ kernel vanishing}
	\end{theorem}
	
	\subsection{Auxiliary Results and Proof of Theorem~\ref{thm:CZ kernel smoothness}}
	\label{subsec: Necessities for Thm4}
	
	In this subsection, we present Proposition~\ref{prop:estimation of alpha partial of K_0} and prove Theorem~\ref{thm:CZ kernel smoothness} by using the proposition. In Appendix~\ref{appx : lemmas needed in Sec2}, we place some lemmas used to show Proposition~\ref{prop:estimation of alpha partial of K_0} and Theorem~\ref{thm:CZ kernel smoothness}. 
	
	We define the series \(K_0\) as
	\[K_0(x,y) := \sum\limits_{k\in \mathbb{Z}} \psi(x-k) \overline{\phi(y-k)},\]
	so that the wavelet frame kernel \(K_{\psi,\phi}\) in (\ref{WaveletFrameKernel}) can be written as
	\[K_{\psi,\phi}(x,y) = \sum\limits_{j \in\mathbb{Z}} A^{j} K_0(A^j x, A^j y)\]
	with a dilation factor \(A>1\).
	
	The following proposition gives the estimation for \(\partial^\alpha_y K_0(x,y)\), and it is a generalization of Lemma 4.7 and 4.8 in \cite{[BuiLaug]WFBonLH}, which handles the case \(p=1\). Our proof of Proposition \ref{prop:estimation of alpha partial of K_0} is similar to the proof of Lemma 4.7 and 4.8 in \cite{[BuiLaug]WFBonLH}, but shows how to handle the case \(0 < p <1\) properly. We use Lemma~\ref{lemma:W^{n,1}} and \ref{lemma:AbsolutelyConvergenceContinuityK0} in Appendix~\ref{appx : lemmas needed in Sec2} to show the proposition.
	
	\begin{proposition}
		Let \(0<p\leq1\) be fixed and let \(\Np = \lfloor 1/p-1 \rfloor\). Assume that \(\psi, \phi \in L^2 \). For any fixed \(0 \leq \alpha \leq \Np +1\), suppose that \(\widehat{\psi},\widehat{\phi},  \cdots, \xi^\alpha\widehat{\phi} \in W^{2,1}\) and \(\widehat{\psi}, \xi^\alpha \widehat{\phi}\in W^{\alpha+2,2}\), and that \(\sigma_\alpha(\psi,\phi)\) and \(\tau_\alpha(\psi,\phi)\) in (\ref{eq:sigma_alpha}) and (\ref{eq:tau_alpha}), respectively, are finite. Then for all \(x,y\in\mathbb{R}\) such that \(x\neq y\), we have
		\[\lvert \partial^\alpha_y K_0(x,y)\rvert \leq \min \left\{ \sigma_\alpha(\psi,\phi), \frac{\tau_\alpha(\psi,\phi)}{\lvert x-y \rvert ^{\alpha+2}} \right\}.\]
		\label{prop:estimation of alpha partial of K_0}
	\end{proposition}
	
	\begin{proof}
		Let us fix \(0 < p \leq 1\) and \(0\leq \alpha \leq \Np+1\). From the assumptions \(\widehat{\psi},\widehat{\phi},\cdots, \xi^\alpha\widehat{\phi} \in W^{2,1}\), we obtain the continuity of \(\psi, \phi, \cdots, \phi^{(\alpha)}\) and decay conditions \(\lvert \psi(x) \rvert, \lvert \phi(x) \rvert, \cdots, \lvert \phi^{(\alpha)}(x) \rvert \lesssim 1/(1+\lvert x \rvert^2)\) by Lemma~\ref{lemma:W^{n,1}}. Hence, by Lemma~\ref{lemma:AbsolutelyConvergenceContinuityK0}, we can write \(\partial_y^\alpha K_0(x,y)\) as a converging series. Then, we define a periodic function \(F(t)\) with period \(1\) as
		\begin{equation}
			F(t) := \partial^\alpha_y K_0(x+t,y+t) = \sum\limits_{k\in \mathbb{Z}} \psi(x+t-k) \,\overline{\phi^{(\alpha)}(y+t-k)}.
			\label{B-PeriodicF(t)}
		\end{equation}
		Since \(\psi,\phi,\cdots, \phi^{(\alpha)}\) are continuous, \(F\) is continuous, by Lemma \ref{lemma:AbsolutelyConvergenceContinuityK0} again.  
		For \(l\in \mathbb{Z}\), the \(l\)-th Fourier coefficient of \(F\) is given as 
		\begin{equation*}
			c_F(l) = \displaystyle\int_0^1 \sum\limits_{k\in \mathbb{Z}} \psi(x+t-k) \, \overline{\phi^{(\alpha)}(y+t-k)} e^{- 2\pi i t l} dt .
		\end{equation*}
		Since \(\phi^{(\alpha)}\) is bounded and \(\lvert \psi(x) \rvert \lesssim 1/(1+\lvert x \rvert^2)\), the series in the integrand converges uniformly. 
		So, we have 
		\begin{eqnarray*}
			c_F(l) & = & \int_{\mathbb{R}} \psi(x+t) \,\overline{\phi^{(\alpha)}(y+t)} e^{-2\pi i t l} dt\\ 
			& = & (-2\pi i)^{\alpha} e^{2\pi i x l} \displaystyle\int e^{2\pi i \xi (x-y)} \xi^\alpha \overline{\widehat{\phi}(\xi)} \widehat{\psi}\left(\xi+ l\right)   d\xi. 
		\end{eqnarray*}
		By summing the absolute value of the coefficients \(c_F(l)\) over \(l\in\mathbb{Z}\), we have
		\[
		\sum\limits_{l\in\mathbb{Z}} \lvert c_F(l)  \rvert \leq (2\pi)^\alpha \sum\limits_{l\in\mathbb{Z}} \left\| \xi^\alpha \overline{\widehat{\phi} (\cdot)} \widehat{\psi} \left(\cdot+ l \right) \right\|_{L^1} = \sigma_\alpha(\psi,\phi)<\infty.
		\]
		Also, one can show that \(\lvert F(t) \rvert \lesssim 1/(1+\lvert t \rvert^2)\) holds. Since \(F\) is continuous with \(\sum_{l\in\mathbb{Z}} \lvert c_F(l)\rvert<\infty\), the Fourier series of \(F\) converges pointwise and
		\(F(0) =\sum_{l\in\mathbb{Z}} c_F(l)\). Then, we have \(\partial^\alpha_y K_0(x,y) = \sum_{l\in\mathbb{Z}} c_F(l)\) from (\ref{B-PeriodicF(t)}).
		Thus we have an estimation for the absolute value of  \(\partial_y^\alpha K_0\) as
		\begin{equation*}
			\lvert \partial_y^\alpha K_0(x,y) \rvert \leq \sum\limits_{l\in\mathbb{Z}} \lvert c_F(l) \rvert \leq (2\pi)^{\alpha} \sum\limits_{l\in\mathbb{Z}} \left\| \xi^\alpha \overline{\widehat{\phi}(\cdot)} \widehat{\psi}\left(\cdot + l \right) \right\|_{L^1},
		\end{equation*}
		where the last inequality comes from the above inequality. 
		
		Since \(\xi^\alpha \widehat{\phi} , \widehat{\psi} \in W^{\alpha+2,2}\), we can use the integration by parts \((\alpha+2)\)-times for each \(c_F(l)\) in \(\partial_y^\alpha K_0(x,y) = \sum_{l\in\mathbb{Z}} c_F(l)\). This gives another estimation for the absolute value of \(\partial_y^\alpha K_0\) as
		\begin{equation*}
			\lvert \partial^\alpha_y K_0(x,y)\rvert   \leq   \left( \frac{1}{4\pi^2} \sum\limits_{l\in\mathbb{Z}} \left\| \left( \xi^\alpha \overline{\widehat{\phi}(\cdot)} \widehat{\psi}\left(\cdot+ l \right)   \right)^{(\alpha+2)} \right\|_{L^ 1 }\right) \frac{1}{\lvert x-y \rvert^{\alpha+2}}
		\end{equation*}
		
		Therefore, for each \(0 \leq \alpha \leq \Np+1\), we conclude that
		\[\lvert \partial^\alpha_y K_0(x,y)\rvert \leq \min \left\{ \sigma_\alpha(\psi,\phi), \frac{\tau_\alpha(\psi,\phi)}{\lvert x-y\rvert^{\alpha+2}} \right\} .\]
	\end{proof}
	
	Based on the estimation on \(\partial^\alpha_y K_0(x,y)\) shown in Proposition~\ref{prop:estimation of alpha partial of K_0}, we prove Theorem~\ref{thm:CZ kernel smoothness} by additionally using Lemma~\ref{lemma:W^{n,1}}, \ref{lemma:appendix1} and \ref{lemma:AbsoluteConvergenceEstimationK} in Appendix~\ref{appx : lemmas needed in Sec2}.
	
	\begin{proof}[Proof of Theorem~\ref{thm:CZ kernel smoothness}]
		We first show that, for each fixed \(0\leq \alpha \leq \N +1\), if \(\widehat{\psi}, \widehat{\phi},  \cdots, \xi^\alpha\widehat{\phi} \in W^{\alpha+2,1}\) and \(\widehat{\psi}, \xi^\alpha \widehat{\phi}\in W^{\alpha+2,2}\), then for all \(x \neq y\), 
			\[\lvert \partial^\alpha_y K_{\psi,\phi}(x,y) \rvert \leq \frac{\CZ_{\alpha}(\psi,\phi)}{\lvert x-y \rvert^{\alpha+1} }.\]
			To see this, we first invoke Lemma \ref{lemma:W^{n,1}} and obtain the decay conditions  \(\lvert \psi(x) \rvert, \lvert \phi(x) \rvert, \cdots, \lvert \phi^{(\alpha)}(x)\rvert\) \( \lesssim 1/(1+\lvert x\rvert^{\alpha+2})\) and the continuity of \(\psi, \phi, \cdots, \phi^{(\alpha)}\) from the assumptions \(\widehat{\psi}, \widehat{\phi}, \cdots, \xi^\alpha \widehat{\phi}\in W^{\alpha+2,1}\). Then, by  Lemma~\ref{lemma:AbsoluteConvergenceEstimationK}, \(\partial_y^\alpha K_{\psi,\phi}(x,y)\) exists and can be written as \(\partial^\alpha_y K_{\psi,\phi}(x,y)  =  \sum_{j\in\mathbb{Z}} A^{j(\alpha+1)} \partial^\alpha_y K_0(A^j x, A^j y)\). Next, we use Proposition \ref{prop:estimation of alpha partial of K_0} to get \(\lvert \partial^\alpha_y K_0(x,y)\rvert \leq \min \left\{ \sigma_\alpha(\psi,\phi), \tau_\alpha(\psi,\phi)/{\lvert x-y \rvert^{\alpha+2}} \right\}\) from the assumptions \(\widehat{\psi}, \xi^\alpha \widehat{\phi} \in W^{\alpha+2,2}\). By taking \(\sigma = \sigma_\alpha(\psi,\phi), \,\tau = \tau_\alpha(\psi,\phi)\) and \(l=\alpha\) in Lemma~\ref{lemma:appendix1}, we obtain the above bound for \(\lvert \partial_y^\alpha K_{\psi,\phi} (x,y) \rvert\).
			
			Now, since the given regularity assumptions of Theorem~\ref{thm:CZ kernel smoothness} imply the above regularity assumptions for every \(0 \leq \alpha \leq \N+1\), the desired bound of Theorem~\ref{thm:CZ kernel smoothness} is obtained by taking the maximum value of \(\CZ_\alpha(\psi,\phi)\) over \(0 \leq \alpha \leq \N+1\), and this completes the proof.
	\end{proof}
	
	\subsection{Auxiliary Results and Proof of Theorem~\ref{thm:CZ kernel vanishing}}
	\label{subsec: Necessities for Thm5}
	
	In this subsection, we present Proposition~\ref{prop:SynthesisVanishingBdd} and \ref{prop:AnalysisVanishingBdd}, and prove Theorem~\ref{thm:CZ kernel vanishing} using them. A more technical lemma for showing Proposition~\ref{prop:AnalysisVanishingBdd} is in Appendix~\ref{appx : lemmas needed in Sec2}.
	
	We start by defining some spaces and operators and by stating propositions. Let \(0<p \leq1\) and \(\Np=\left\lfloor 1/p -1 \right\rfloor\) as before. Let us define the space \(K^{\Np+1,2}\), which can be thought of as the Fourier transform of the Sobolev space \(W^{\Np+1,2}\), as follows:
	\begin{equation*}
		K^{\Np+1,2} := \left\{f\in L^2: \left\| f \right\|_{K^{\Np+1,2}}:=\int (1+x^{2\Np+2}) \lvert f(x) \rvert ^2 dx< \infty\right\}.
	\end{equation*}
	We further define the space \(K^{\Np+1,2}_*\) with the vanishing moment condition
	\begin{equation*}
		K^{\Np+1,2}_* := \left\{f\in K^{\Np+1,2}: \int f(x) x^\alpha dx = 0, \,\forall \, 0 \leq \alpha \leq \Np\right\},
	\end{equation*}
	and the corresponding sequence space \(l^{\Np+1,2}\) as the set of all sequences \(c=\{c_{jk}\}\) such that \[\left\| c \right\|_{l^{\Np+1,2}}^2 := \sum\limits_{j,k\in\mathbb{Z}} \left( 1+ A^{-2j(\Np+1)} \left(1+k^{2(\Np+1)}\right)\right) c_{jk}^2 <\infty\]
	for \(A>1\). 
	These spaces are the generalization of the corresponding spaces studied in \cite{[BuiLaug]WLPMHC} for the case when \(\Np=0\).
	
	We now recall the wavelet analysis and synthesis operators for \(L^2\) and \(l^2:=l^2(\mathbb{Z}\times \mathbb{Z})\).
	
	\begin{definition}
		Let \(\psi,\phi\in L^2\) satisfy the decay condition (\ref{eq:BLdecay}).
		\begin{enumerate}[(a)]
			
			\item The \textit{(wavelet) analysis operator} \(t := t_\phi:L^2 \to l^2\) (associated with an analyzer \(\phi\)) is a bounded linear map defined by
			\[f \mapsto t(f) = \{\left<f,\phi_{jk}\right>\}_{j,k\in\mathbb{Z}}.\]
			\item The \textit{(wavelet) synthesis operator} \(s := s_\psi:l^2 \to L^2\) (associated with a synthesizer \(\psi\)) is a bounded linear map defined by
			\[c=\{c_{jk}\} \mapsto s(c)=\sum\limits_{j,k\in\mathbb{Z}} c_{jk} \psi_{jk}.\]
		\end{enumerate}
		\label{def:synthesis,analysis operator}		
	\end{definition}

	The boundedness of these operators is proved in \cite{[BuiLaug]FSF,[BuiLaug]WLPMHC}, used in our proof of Proposition~\ref{prop:SynthesisVanishingBdd} and \ref{prop:AnalysisVanishingBdd} below, and stated separately in the following lemmas for a clearer presentation.
	
	\begin{lemma}[\cite{[BuiLaug]FSF,[BuiLaug]WLPMHC}]
		Assume that \(\Phi \in L^2 \) satisfies the decay condition (\ref{eq:BLdecay}). Let \(\phi(x) = \widecheck{\Phi} (-x)\). Then the analysis operator \(t: L^2 \to l^2\) is bounded and linear.
		\label{lemma: analyzer L2 bddness}
	\end{lemma}

	\begin{lemma}[\cite{[BuiLaug]FSF,[BuiLaug]WLPMHC}]
		Assume that \(\Psi \in L^2\) satisfies the decay condition (\ref{eq:BLdecay}). Let \(\psi(x) = \widecheck{\Psi}(-x)\). Then the synthesis operator \(s:l^2 \to L^2\) is bounded and linear, with unconditional convergence of the series.
		\label{lemma: synthesis L2 bddness}
	\end{lemma}

	The following proposition shows the synthesis operator \(s\) restricted to \(l^{\Np+1,2}\) is well-defined. It is proved using Lemma~\ref{lemma: synthesis L2 bddness}.
	
	\begin{proposition}
		Let \(0<p \leq 1\) be fixed and let \(\N:=\Np\). Assume that \(\psi \in L^2\) satisfies the following conditions: 
		\begin{enumerate}[(a)]
			\item \(\widehat{\psi} \in W^{\N+1,2}\), and \( \int \psi(x) x^\beta dx= 0\), \(\forall \, 0 \leq \beta \leq \N\),
			\item \(\displaystyle \lvert\widehat{\psi}^{(\N+1)}(\xi)\rvert \lesssim \begin{cases}\lvert \xi \rvert^\epsilon, & \lvert \xi \rvert \leq 1,
				\lvert \xi \rvert^{-\epsilon-\N-3/2}, & \lvert \xi \rvert \geq 1, \end{cases}\) \quad for some \(\epsilon>0.\)
		\end{enumerate}
		Then the restriction \(s \vert_{l^{\N+1,2}} :l^{\N+1,2} \to K^{\N+1,2}_{*}\)of the synthesis operator is well-defined, bounded and linear, with unconditional convergence of the series.
		\label{prop:SynthesisVanishingBdd}
	\end{proposition}
	
	\begin{proof}
		Let \(c\in l^{\N+1,2}\). In order to show \(s(c) \in K^{\N+1,2}\), it suffices to show that \(\left\| s(c) \right\|_{L^2}\) , \(\left\| x^{\N+1} s(c) \right\|_{L^2} <\infty\). Let \(\Psi(-\xi):=\widehat{\psi}(\xi)\). Then, since \(\Psi \in L^2\) and it has the decay condition (\ref{eq:BLdecay}), we get \(s(c) \in L^2\) by Lemma \ref{lemma: synthesis L2 bddness} immediately. 
		Next, since 
		\(
		x^{\N+1} = A^{-j(\N+1)} \sum_{i=0}^{\N+1} \binom{\N+1}{i} (A^jx-k)^{i} k^{\N+1-i},
		\)
		we have
		\begin{equation*}
			x^{\N+1} s(c) =  \sum\limits_{i=0}^{\N+1} \binom{\N+1}{i} \sum\limits_{j,k\in\mathbb{Z}} \left( A^{-j(\N+1)} k^{\N+1-i} c_{jk} \right) (\eta_i)_{jk}(x),
		\end{equation*}
		where \(\eta_i := x^i\psi\). By taking \(\Psi\) in Lemma~\ref{lemma: synthesis L2 bddness} as \(\Psi^{(i)}\) and noting that \(\Psi^{(i)}=\widehat{\eta_i}\), we have
		\begin{equation*}
			\left\| \sum\limits_{j,k\in\mathbb{Z}} \left(A^{-j(\N+1)} k^{\N+1-i} c_{jk}\right) (\eta_i)_{jk} \right\|_{L^2}  \lesssim \left\| A^{-j(\N+1)} k^{\N+1-i} c_{jk} \right\|_{l^2}. 
		\end{equation*}
		Since the right-hand side of the above inequality is bounded by \(\left\| c \right\|_{l^{\N+1,2}}\), we get
		\(
		\left\| x^{\N+1} s(c) \right\|_{L^2} \lesssim \left\| c \right\|_{l^{\N+1,2}}
		\).
		
		Finally, fix \(0 \leq \alpha \leq \N\). Then, since \(\left\| x^{\alpha} s(c)\right\|_{L^1} \lesssim \left\| s(c) \right\|_{K^{\N+1,2}}<\infty\) 
		from the assumption that \(\int \psi(x) x^\beta dx =0 \) for all \(0 \leq \beta \leq \N\), we see 
		that \(\int s(c)(x) x^\alpha dx = 0\). Therefore, \(s(c) \in K^{\N+1,2}_*\), which completes the proof.
	\end{proof}
	
	The next proposition shows the analysis operator \(t\) restricted to \(K^{\Np+1,2}_{*}\) is well-defined. It is proved by using Lemma~\ref{lemma: analyzer L2 bddness}, and Lemma~\ref{lemma:Generalized Hardy's Inequality} in Appendix~\ref{appx : lemmas needed in Sec2}.

	\begin{proposition}
		Let \(0<p \leq 1\) be fixed and let \(\N:=\Np\).  Assume that \(\phi\in L^2\) satisfies the following conditions: 
		\begin{enumerate}[(a)]
			\item \(\widehat{\phi} \in W^{\N+1,2}\), and \( \int \phi(x) x^\beta dx = 0\), \(\forall \, 0 \leq \beta \leq \N\),
			\item \(\displaystyle \lvert\widehat{\phi}^{(\N+1)}(\xi)\rvert \lesssim \begin{cases}\lvert \xi \rvert^\epsilon, & \lvert\xi\rvert \leq 1,
				\lvert\xi \rvert^{-\epsilon-2\N-5/2}, & \lvert\xi \rvert \geq 1,\end{cases}\) \quad for some \(\epsilon>0\).
		\end{enumerate}
		Then the restriction \(t \vert_{K^{\N+1,2}} :K^{\N+1,2}_* \to l^{\N+1,2}\) of the analysis operator is well-defined, bounded and linear.
		\label{prop:AnalysisVanishingBdd}
	\end{proposition}
	
	\begin{proof}
		Let \(f\in K^{\N+1,2}_*\). To show \(t(f) = \{\left<f, \phi_{jk} \right>\} \in l^{\N+1,2}\), it suffices to show
		\begin{equation}
			\sum\limits_{j,k\in\mathbb{Z}} \lvert \left< f, \phi_{jk} \right> \rvert ^2 \lesssim \left\| f \right\|_{L^2}^2,
			\label{eq:AnalysisVanishingCond1}
		\end{equation}
		\begin{equation}
			\sum\limits_{j,k\in\mathbb{Z}} A^{-2j(\N+1)} \lvert \left< f, \phi_{jk} \right> \rvert^2 \lesssim \left\| x^{\N+1} f \right\|_{L^2}^2,\quad \text{and}
			\label{eq:AnalysisVanishingCond2}
		\end{equation}
		\begin{equation}
			\sum\limits_{j,k\in\mathbb{Z}} A^{-2j(\N+1)} k^{2(\N+1)} \lvert \left<f,\phi_{jk} \right> \rvert^2 \lesssim \left\| x^{\N+1} f \right\|_{L^2}^2.
			\label{eq:AnalysisVanishingCond3}
		\end{equation}
		Let \( \Phi(-\xi):=\widehat{\phi}(\xi)\). 
		Since \(\Phi \in L^2\) with the decay condition (\ref{eq:BLdecay}), the inequality (\ref{eq:AnalysisVanishingCond1}) is immediate from Lemma \ref{lemma: analyzer L2 bddness}.
		
		Next, let \(F(\xi):=\widecheck{f}(\xi)\). Since \(F\in W^{\N+1,2}\), \(\left\|F/ \xi^{\N+1} \right\|_{L^2} \leq c_1 \left\| F^{(\N+1)} \right\|_{L^2}\) by Lemma~\ref{lemma:Generalized Hardy's Inequality}, and this implies \(F/\xi^{\N+1} \in L^2\). 
		Then, we have \(g:=\widehat{\left( F/\xi^{\N+1} \right)} \in L^2\) and \(g^{(\N+1)} = f\) holds weakly. From 
		\begin{equation*}
			A^{-j(\N+1)} \left<f,\phi_{jk} \right> = (-1)^{\N+1} \left< g, \left(\phi^{(\N+1)}\right)_{jk} \right>
		\end{equation*}
		and by Lemma~\ref{lemma: analyzer L2 bddness} with \(\Phi\) there as \(\xi^{\N+1} \Phi\), we get the inequality (\ref{eq:AnalysisVanishingCond2}) since 
		\begin{equation*}
			\sum\limits_{j,k\in\mathbb{Z}} A^{-2j(\N+1)} \lvert \left<f,\phi_{jk} \right> \rvert^2 \leq c_2 \lVert g \rVert_{L^2}^2\le c_1^2 c_2 \lVert x^{\N+1} f \rVert_{L^2}^2.
		\end{equation*}
		
		Finally, from
		\begin{equation*}
			k^{\N+1} = (A^j x -(A^j x - k))^{\N+1} = \sum\limits_{i=0}^{\N+1} \binom{\N+1}{i} (A^jx)^{\N+1-i} (-(A^jx-k))^{i},
		\end{equation*}		
		to get the inequality (\ref{eq:AnalysisVanishingCond3}), it suffices to show that for any fixed \(0 \leq i \leq \N+1\), 
		\begin{eqnarray}
			\sum\limits_{j,k\in\mathbb{Z}} A^{-2j(\N+1)} (A^j x)^{2(\N+1-i)} (A^jx-k)^{2i}\lvert \left<f, \phi_{jk} \right> \rvert^2 \leq \widetilde{c}_i\left\| x^{\N+1} f \right\|_{L^2}^2
			\label{eq:AnalysisVanishingCond3_auxiliary}
		\end{eqnarray}
		for some \(\widetilde{c}_i>0\). For \(i=0\), since \(\Phi \in L^2\) and it has the decay condition (\ref{eq:BLdecay}), the above inequality holds due to Lemma~\ref{lemma: analyzer L2 bddness}. So, we may assume that \(0 < i \leq \N+1\) and fix \(i\). Note that
		\begin{equation*}
			A^{-j(\N+1)} (A^jx)^{\N+1-i} (A^jx-k)^i \left<f, \phi_{jk} \right> = A^{-ij} \left< x^{\N+1-i} f, (A^jx-k)^i \phi_{jk} \right>.
		\end{equation*}
		Also, note that \(x^{\N+1-i} f \in K^{i,2}\). Let \(F(\xi):= \widecheck{\left(x^{\N+1-i} f \right)}(\xi)\). Then \(F \in W^{i,2}\) and by Lemma~\ref{lemma:Generalized Hardy's Inequality}, \(\left\| F/\xi^i \right\|_{L^2} \leq c_3 \left\| F^{(i)} \right\|_{L^2}\). Thus, \(F/\xi^i \in L^2\). Then, we have \(g := \widehat{F/\xi^i} \in L^2\) and \(g^{(i)} = x^{\N+1-i} f\) holds weakly. Let \(\eta := x^i \phi\). Then,
		\begin{equation*}
			A^{-ij} \left< x^{\N+1-i} f , (A^jx-k)^i \phi_{jk} \right> = A^{-ij} \left< g^{(i)}, \eta_{jk} \right> = (-1)^i \left<g, (\eta^{(i)})_{jk} \right>.
		\end{equation*}
		Thus, by Lemma~\ref{lemma: analyzer L2 bddness} with \(\Phi\) there as \( \xi^i \Phi^{(i)}\) and by noting that \( \xi^i \Phi^{(i)}=\widehat{\eta^{(i)}}\), we see that the left-hand side of (\ref{eq:AnalysisVanishingCond3_auxiliary}) is bounded by
		\begin{equation*}
			\sum\limits_{j,k\in\mathbb{Z}} \Bigl\lvert \left<g, (\eta^{(i)})_{jk} \right>\Bigr\rvert ^2 \leq c_4\left\| g \right\|_{L^2}^2\le c_3^2 c_4 \left\| x^{\N+1} f \right\|_{L^2}^2. 
		\end{equation*}
		Therefore, we have the inequality (\ref{eq:AnalysisVanishingCond3_auxiliary}) for fixed \(i\), with \(\widetilde{c}_i=c_3^2 c_4\). This completes the proof.
	\end{proof}

	By using Proposition~\ref{prop:SynthesisVanishingBdd} and \ref{prop:AnalysisVanishingBdd}, let us prove Theorem~\ref{thm:CZ kernel vanishing}. 
	
	\begin{proof}[Proof of Theorem~\ref{thm:CZ kernel vanishing}]
		Let \(\psi\) and \(\phi\) satisfy all the assumptions in Theorem~\ref{thm:CZ kernel vanishing}. Since all assumptions are the union of all the assumptions in Proposition~\ref{prop:SynthesisVanishingBdd} and \ref{prop:AnalysisVanishingBdd}, we see that the composition \(s\circ t \) is bounded and linear on \(K^{\N+1,2}_*\). Since the wavelet frame operator \(U\) satisfies \(Uf = (s \circ t)(f)\) for every \(f\in K^{\N+1,2}_*\), \(U\) is bounded on \(K^{\N+1,2}_*\). 
		
		Let \(f\) be a function in \(L^2\) with compact support, and satisfy \(\int f(x) x^\gamma dx = 0\) for every \(0 \leq \gamma \leq \N\). Then \(f\in K^{\N+1,2}_*\), and thus \(Uf\in K^{\N+1,2}_*\) as well. Let \(0\leq \alpha \leq \N\) be fixed. Since 
		\begin{equation*}
			\lVert x^\alpha Uf \rVert_{L^1} \leq \left(\int \frac{x^{2\alpha}}{1+x^{2\N+2}} dx\right)^{1/2} \left\| \sqrt{1+x^{2\N+2}} Uf \right\|_{L^2} \lesssim \left\| Uf \right\|_{K^{\N+1,2}}<\infty,
		\end{equation*}
		by the assumption that \(\int \psi(x) x^\beta dx =0 \) for all \(0 \leq \beta \leq \N\), we have
		\begin{equation*}
			\int Uf(x) x^\alpha dx = \int f(y) \left( \sum\limits_{j,k\in\mathbb{Z}} A^j \overline{\phi(A^j y -k)} \int \psi(A^jx-k) x^\alpha dx \right) dy =0,
		\end{equation*}
		which completes the proof.
	\end{proof}

	\section{Boundedness of Calder\'{o}n-Zygmund Operator on $H^p(\mathbb{R})$ with an Explicit Bound}
	\label{sec:Bdd of U on Hardy}
	
	\subsection{Main Result}
	\label{subsec:bdd of CZO}
	
	Among the known equivalent definitions of Hardy space \(H^p=H^p(\mathbb{R})\), we adopt the definition of Hardy space with the building block called ``atom'' \cite{[GarciaFrancia]WNI,[GilbertHogan]SmoothMolecularDecomp,[Grafakos]MFA,[Meyer]CZM}. Recall that \(\Np= \left\lfloor 1/p-1 \right\rfloor\). 
	
	\begin{definition}
		For \(0<p\leq 1\), a real-valued function \(\atom \in L^2\) is called a \textit{\((p,2)\)-atom} (or simply, an \textit{atom}) if it satisfies the following conditions:
		\begin{enumerate} [(a)]
			\item supp \(\atom \subset I\) for some bounded interval \(I\),
			\item \(\left\|\atom \right\|_{L^2} \leq \lvert I \rvert^{1/2-1/p}\), and 
			\item \(\displaystyle\int \atom(x) x^\alpha dx = 0\), \(\forall \, 0\leq \alpha \leq \Np \).
		\end{enumerate}
	\label{def:atom}
	\end{definition}
	
	\begin{definition}
		For \(0<p \leq1\), the \textit{(atomic) Hardy space} \(H^p\) is defined by
		\[
		\begin{array}{lclclcl}
			H^p &:= & \{ f\in \mathcal{S}' & : & f= \sum_{k=0}^{\infty} \lambda_k \atom_k \text{ (in the sense of tempered distribution)}\text{ for some }\\
			[8pt] 
			& & & & (p,2)\text{-atoms } \atom_k \text{ and }\sum_{k=0}^{\infty} \lvert \lambda_k \rvert ^p <\infty\},
		\end{array}
		\]
		and the quasi-norm \(\left\| f\right\|_{H^p}\) is defined as \( \inf \Biggl\{ \left( \sum_{k=0}^{\infty} \lvert \lambda_k \rvert^p\right)^{1/p} \Biggr\}\), where the infimum is taken over all possible atomic decompositions of \(f\in H^p\).
		\label{def:AtomicHardy}
	\end{definition}
	
	The Hardy space \(H^p\) is a complete space with \(\left\| \cdot \right\|_{H^p}\) for \(0<p \leq 1\) \cite{[Bownik]AniHardy,[Grafakos]MFA}. We will show that the Calder\'{o}n-Zygmund operator \(Z\) (see Definition~\ref{def: Calderon-Zygmund}) has a bounded extension to \(H^p\). This is an extension of the result in \(H^1\) space studied in \cite{[BuiLaug]WFBonLH}. 
	
	Before stating our \(H^p\) extension result, we introduce some numbers used in its statement and throughout the paper. Let \(0<p \leq 1\) be fixed and let \(\N:=\N_p\). Let  \(b\) denote any number larger than \(2/p\) and let \(\delta\) be defined as
	\begin{equation}
		\delta:= \delta(b): =  0.5\cdot \left((2 \N+3)/(2b) + \sqrt{4+((2\N+3)/(2b))^2}\right).
		\label{eq:b and delta}
	\end{equation}
		We also reserve the letter \(\Calphamax\) to denote 
				\begin{equation}
			\Calphamax:= \Calphamax(\N):= \max_{0 \leq \alpha \leq \N} \Calpha
			\label{eqn : g}
		\end{equation}
		where \(\Calpha\) are the values described in the following lemma from \cite{[GarciaFrancia]WNI}.
		
		\begin{lemma}[\cite{[GarciaFrancia]WNI}]
			Let \(N\in \mathbb{N}\) and \(R>0\) be fixed. Let \(0 \le \alpha \le N\) be a fixed integer. For each \(k\in \mathbb{N}_0\), there exists a unique polynomial $g_\alpha^k(x)$ of degree at most \(N\) such that
			\begin{align}
			\frac{1}{\lvert E_k\rvert}\int_{E_k} g_\alpha^k(x) x^\beta dx=
			\begin{cases}
			1, \quad &\text{if }\beta=\alpha, \\
			0, \quad &\text{if } 0 \le \beta \le N, \; \beta \neq \alpha,
			\end{cases}
			\label{phiConditionEk}
			\end{align}
			where \(E_0:=\{x\in \mathbb{R} : \lvert x \rvert \le R\}\) and 
			\begin{equation}
			E_k:=\{x\in \mathbb{R} : 2^{k-1}R < \lvert x \rvert \le 2^k R\}, \quad k=1,2,\cdots.
			\label{eq:E_k}
			\end{equation}
			Furthermore, for $G_\alpha^k(x):=g_\alpha^k(x)\chi_{E_k}(x)$, there exists a constant \(\Calpha\) independent of \(k\) and \(R\) such that
			\begin{equation}
			\lvert G_\alpha^k(x) \rvert \le \Calpha (2^k R)^{-\alpha}, \quad \text{for every } x \in \mathbb{R}.
			\label{eq:Calpha}
			\end{equation}
			\label{lemma : polynomial G}
		\end{lemma}

		Clearly, \(\Calphamax\) depends on \(\N\), and is independent of other parameters including~\(R\) in Lemma~\ref{lemma : polynomial G} when \(\N\) is fixed. Some more details about values \(\Calpha\) and \(\Calphamax\) are given via an example in Appendix~\ref{appx : example of Lemma9}. 
		

	\begin{theorem}
		Let \(0<p \leq 1\) be fixed and let \(\N:= \Np\). Suppose \(Z\) is a Calder\'{o}n-Zygmund operator with a constant \(\gCZ\). Then for any fixed \(b>2/p\) and \(\zeta \geq \delta\) with \(\delta=\delta(b)\) as in (\ref{eq:b and delta}), \(Z\) has a bounded extension to \(H^p\) satisfying
		\[
		\left\| Z\right\|_{H^p\to H^p}^p  \leq  C_1 \zeta^{p\left(1/p-1/2\right)} \left\| Z \right\|_{L^2\to L^2}^p + C_2 \gCZ^p,
		\]
		where \(C_1\) and  \(C_2:=C_2(b,\zeta)\) are given as follows:
		\[
		C_1 := \left( 1+ \Calphamax \left( \N +1 \right)\right)^p,
		\]
		\begin{eqnarray*}
			C_2(b,\zeta)& := & \left\{\frac{2^{\N+3}}{(\N+1)! \sqrt{2\N+3}}\right\}^p \left\{ \frac{1}{2} \left( \frac{1}{(\zeta + 1)^{2\N+3}} + \frac{1}{(\zeta -1)^{2\N +3}}\right) \zeta^{\frac{2}{p}-1}\right\}^{p/2} \nonumber \\
			&  &  \times \left\{ 2\cdot (2b-1)^{-p/2} C_1 + 3 \Calphamax^p \sum\limits_{0 \leq \alpha \leq \N }  \left(\frac{2 + 2^{-\alpha}}{b-\alpha-1} \right)^p\right\}.
		\end{eqnarray*}
		Here, \(\Calphamax\) is a constant defined as in~(\ref{eqn : g}).
		\label{thm:CZO has Hp bddness}
	\end{theorem}
				
	\begin{nonumremark} 
In Theorem~\ref{thm:CZO has Hp bddness}, when \(p\), \(\N\) and \(\Calphamax\) are fixed, the constant \(C_1\) is a fixed number containing no variables, and \(C_2=C_2(b,\zeta)\) depends only on parameters~\(b\) and~\(\zeta\). Hence, the above bound 
			\[
				C_1 \zeta^{p\left(1/p-1/2\right)} \left\| Z \right\|_{L^2\to L^2}^p + C_2 \gCZ^p
			\]				
can be improved by taking sufficiently large \(b\). To see this, note that as the parameter \(b\) increases (and \(\zeta\) is fixed), the value of \(C_2\) decreases, so the second term, i.e. \(C_2 \gCZ^p\), decreases, whereas the first term stays the same. 
		\qed
	\end{nonumremark}
			

	\subsection{Auxiliary Results and Proof of Theorem~\ref{thm:CZO has Hp bddness}}
	\label{subsec: Necessities for Thm6}

	In this subsection, we present two propositions and the proof of Theorem~\ref{thm:CZO has Hp bddness} based on these.
	
	To show the boundedness of a Calder\'{o}n-Zygmund operator on \(H^p\) for any fixed \(0<p \leq 1\), we follow a classical approach for the atomic Hardy space \cite{[Bownik]AniHardy,[GarciaFrancia]WNI}. This approach consists of three steps. First, one shows that the operator maps atoms to ``molecules'' and then, in the second step, shows that every ``molecule'' belongs to the Hardy space. In the final step, by using the results from the previous steps and an atomic decomposition, it is shown that a Calder\'{o}n-Zygmund operator is bounded on \(H^p\). We follow this approach in such a way that a bound of the operator norm from \(H^p\) to \(H^p\) is a computable number. For a restricted case of \(p=1\), similar results are obtained in \cite{[BuiLaug]WFBonLH}.
	
	The following proposition corresponds to the first step, which shows that a Calder\'{o}n-Zygmund operator maps an atom to the function, with the specific conditions given below. This function is the so-called ``molecule'' \cite{[GarciaFrancia]WNI,[Kyr]WCMS}.
	
	\begin{proposition}
		Let \(0<p \leq 1\) be fixed and \(\N:=\Np\). Let \(Z\) be a Calder\'{o}n-Zygmund operator with a constant \(\gCZ\). Also, let \(\atom\) be an atom supported in a bounded interval~\(I\) centered at \(y_0\in\mathbb{R}\). Then for any fixed \(b>2/p\) and \(\zeta \geq \delta\) with \(\delta=\delta(b)\) as in (\ref{eq:b and delta}), we have \(Z \atom\in L^2\) and
		\[\lvert (Z\atom)(x)\rvert \leq \frac{C_3 \gCZ \lvert I\rvert^{b-1/p}}{\lvert x-y_0\rvert^b},\quad x\notin\zeta I,\]
		where \(C_3:=C_3(b,\zeta)\) is the constant defined by
		\begin{equation}
			C_3(b,\zeta) := \frac{2^{\N+3}}{(\N+1)! \sqrt{2\N+3}} \left\{\frac{1}{2\zeta}\left(\frac{1}{(\zeta+1)^{2\N +3}} + \frac{1}{(\zeta -1)^{2\N+3}}\right)\right\}^{1/2} \left(\frac{\zeta}{2}\right)^b.
			\label{eq:C_3}
		\end{equation}	
		\label{prop: CZO maps atoms to molecule}
	\end{proposition}

	\begin{proof}
		Fix \(x\notin \zeta I\). Since \(Z\) is a Calder\'{o}n-Zygmund operator defined with kernel \(\K\) and since \(\atom \in L^2\) is an atom having vanishing moments, we have
		\begin{equation*}
			\lvert Z \atom(x) \rvert = \biggl\lvert \int_I \left( \K(x,y) - \sum\limits_{0 \leq \alpha \leq \N} \frac{(y-y_0)^\alpha}{\alpha!} \partial_y^\alpha \K(x,y_0) \right) \atom(y) dy \biggr\rvert.
		\end{equation*}
		Then, \(
		\lvert Z\atom(x) \rvert \leq (1/(\N+1)!)\int_I \lvert \partial_y^{\N+1} \K (x,\tilde{y}) \rvert \lvert y-y_0\rvert^{\N+1} \lvert \atom(y) \rvert dy
		\)
		for some \(\tilde{y}\) between \(y\) and \(y_0\).
		From the condition (\ref{CZCond2}) of the Calder\'{o}n-Zygmund operator and the fact that \(\lvert x-y \rvert \leq 2 \lvert x-\tilde{y} \rvert\), we have, with the constant \(\gCZ\),
		\begin{eqnarray*}
			\lvert Z\atom(x) \rvert & \leq &\frac{2^{\N+2} \gCZ}{(\N+1)!} \displaystyle\int_I \frac{\lvert y-y_0 \rvert^{\N+1}}{\lvert x-y \rvert^{\N+2}} \lvert \atom(y) \rvert dy \\
			& \leq & \frac{2^{\N+2} \gCZ}{(\N+1)!} \left(\int_I \frac{\lvert y-y_0 \rvert ^{2(\N+1)}}{\lvert x-y \rvert^{2(\N+2)}} dy \right)^{1/2} \left\|\atom\right\|_{L^2}.
		\end{eqnarray*}
		Let \(w=2(x-y_0)/\lvert I \rvert\) and \(z=2(y-y_0)/\lvert I \rvert\). Then by using the change of variable \(t=w-z\), the integral in the last term of the above inequality is computed as
		\begin{equation*}
			\int_I \frac{\lvert y-y_0 \rvert^{2\N+2}}{\lvert x-y \rvert^{2\N+4}} dy = \frac{2}{\lvert I \rvert} \displaystyle\int_{-1}^1 \frac{z^{2\N+2}}{(w-z)^{2\N+4}} dz = \frac{2}{\lvert I \rvert} \frac{2}{2\N+3} g\left(\lvert w \rvert \right)
		\end{equation*}
		where \(g(\lvert w \rvert) := (1/(\lvert w \rvert + 1)^{2\N+3} + 1/(\lvert w \rvert-1)^{2\N+3})/2\lvert w \rvert\).
		From this and the fact that \(\left\| \atom \right\|_{L^2} \leq \lvert I \rvert^{1/2 - 1/p}\), we have the estimation 
		\begin{equation*}
			\lvert Z \atom(x) \rvert \leq 2^{\N+2} \gCZ \frac{2\lvert I \rvert^{-1/p}}{(\N+1)! \sqrt{2\N+3}} \sqrt{w^{2b} g(\lvert w \rvert)} \frac{1}{\lvert w \rvert^b}.
			\label{EstimationZaFFFINAL}
		\end{equation*}
		Since \(x\notin \zeta I\), from the definition of \(w\), we have \(\lvert w\rvert \geq \zeta \geq \delta\). Also, since \(s \mapsto s^{2b} g(s)\) is decreasing for \(s \geq \delta\), it follows that \(
		\sqrt{w^{2b} g(\lvert w \rvert)}  \leq \sqrt{\zeta^{2b} g(\zeta)}.
		\)
		By using this inequality, we obtain the stated bound for \(\lvert Z \atom(x) \rvert\).
	\end{proof}
	
	The following proposition corresponds to the second step, which shows that if the function has a sufficient decay with sufficient vanishing moments as given below, it belongs to the \(H^p\) space. The proof uses lemmas in Appendix~\ref{appx: Necessities for Thm6}.
	
	\begin{proposition}
		Let \(0<p \leq 1\) be fixed and \(\N:=\Np\). Suppose \(M\in L^2\) and \(I\) is a bounded interval centered at \(y_0\in \mathbb{R}\). For any fixed \(b>2/p\) and \(\zeta \geq \delta\) where \(\delta= \delta(b)\) is defined in (\ref{eq:b and delta}), if
		\begin{equation}
			\lvert M(x) \rvert \leq \frac{C_M \lvert I \rvert ^{b-1/p}}{\lvert x-y_0 \rvert^b},\quad x\notin \zeta I,
			\label{MBddCondition}
		\end{equation}
		for some \(C_M>0\) and 
		\begin{equation}
			\displaystyle\int_{\mathbb{R}} M(x) x^\alpha dx =0, \quad \forall \, 0 \leq \alpha \leq \N,
			\label{MVanishingCondition}
		\end{equation}
		then \(M\in H^p\) with
		\begin{equation*}
			\left\| M \right\|_{H^p}^p \leq C_1 (\zeta \lvert I \rvert)^{p \left( 1/p - 1/2\right)} \left\| M \right\|_{L^2}^p +  C_M^p C_4
		\end{equation*}
		where \(C_1=\left( 1+ \Calphamax \left( \N +1 \right)\right)^p\), as before, with \(\Calphamax\) as in (\ref{eqn : g}), and \(C_4:=C_4(b,\zeta)\) is defined by 
		\begin{equation}
			C_4(b,\zeta):= 2^{bp}  \zeta^{1-bp} \left( 2\cdot (2b-1)^{-p/2} C_1  + 3 \Calphamax^p \sum\limits_{0 \leq \alpha \leq \N }  \left(\frac{2 + 2^{-\alpha}}{b-\alpha-1} \right)^p \right).
			\label{eq:C_4}
		\end{equation}
		\label{prop: Molecule is in Hp}
	\end{proposition}
	
	\begin{proof}
		We may assume that \(y_0=0\) in (\ref{MBddCondition}) by using an appropriate translation. We use collections of sets \(\{E_k\}_{k\in \mathbb{N}_0}\) and of functions \(\{G_\alpha^k\}_{0 \leq \alpha \leq \N, k\in \mathbb{N}_0}\) as in Lemma~\ref{lemma : polynomial G}, with \(N\) and \(R\) chosen as \(\N\) and \(\zeta \lvert I \rvert/2\), respectively. Let \(M\in L^2\) satisfy all the assumptions in the proposition. For each \(k\in \mathbb{N}_0\), we define two functions needed for the proof by using the set \(E_k\) and functions \(\{G_\alpha^k\}_{0\leq \alpha \leq \N}\) as follows: 		
		\begin{equation}
			M_k(x) :=M(x)\chi_{E_k}(x),
			\label{eqn:M_k}
		\end{equation}
		\begin{equation}
			P_k(x):=\sum\limits_{0 \leq \alpha \leq \N } m_\alpha^k G_\alpha^k(x)
			\label{eqn:P_k},
		\end{equation} 
		where \(m_\alpha^k := (1/\lvert E_k \rvert) \int M_k(x) x^\alpha dx\).
		
		Since \(M_k\) and \(P_k\) are supported in \(E_k\) for each \(k\in \mathbb{N}_0\), and the sets \(\{E_k\}\) are disjoint, \(M\) can be written as 
		\[M(x) = \sum\limits_{k=0}^\infty M_k(x) = \sum\limits_{k=0}^\infty \left( M_k(x)-P_k(x)\right) + \sum\limits_{k=0}^\infty P_k(x). \]
		Since \(\left\| M\right\|_{H^p}^p\le \left\| \sum_{k=0}^\infty \left(M_k-P_k\right)\right\|_{H^p}^p+ \left\| \sum_{k=0}^\infty P_k\right\|_{H^p}^p\),  to obtain \(\left\| M\right\|_{H^p}^p\), we will calculate \(\left\| \sum_{k=0}^\infty \left(M_k-P_k\right)\right\|_{H^p}^p\) and \(\left\| \sum_{k=0}^\infty P_k\right\|_{H^p}^p\)by decomposing them as \(\left(p,2\right)\)-atoms. 
				
		For \(k \in \mathbb{N}_0\), let \(\lambda_k\) satisfy
		\begin{equation}
			\lambda_k\le \begin{cases}
				\left(1+\Calphamax \left( \N +1 \right) \right) (\zeta \lvert I \rvert)^{1/p-1/2}\left\|M \right\|_{L^2} ,  & k=0,\\
				C_M \left(1+\Calphamax \left( \N +1\right)\right)(2^{k-1} \zeta)^{1/p-b} \left(\frac{2^{2b}-2}{2b-1}\right)^{1/2}, & k\ge 1,
			\end{cases}
		\label{eqn:claimresult1}
		\end{equation} 
Then since \((M_k-P_k)/\lambda_k\) is a \(\left(p,2\right)\)-atom by Lemma~\ref{claim:M_k-P_k atom}, we have
		\[
		\left\|\sum\limits_{k=0}^\infty \left(M_k-P_k\right) \right\|_{H^p}^p = \left\| \sum\limits_{k=0}^\infty \lambda_k \left(\frac{M_k-P_k}{\lambda_k}\right) \right\|_{H^p}^p \leq \sum\limits_{k=0}^\infty \lvert \lambda_k \rvert^p.
		\]
		Since \(\sum_{k=0}^\infty 2^{(1-bp)k}=1/(1-2^{1-bp})< 2\) for \(b>2/p\), with (\ref{eqn:claimresult1}), this gives
		\begin{eqnarray}
			\left\|\sum\limits_{k=0}^\infty \left(M_k-P_k\right) \right\|_{H^p}^p & \leq & (1+\Calphamax \left(\N +1 \right))^p  \nonumber \\ 
			&{ \times }& \left( (\zeta \lvert I \rvert)^{1-p/2}\left\|M \right\|_{L^2}^p+2 C_M^p \zeta^{1-bp} 2^{bp} (2b-1)^{-p/2} \right). 
			\label{HpNormofMk-Pk}
		\end{eqnarray}
		
		From the fact that \(G_\alpha^k\) is supported in \(E_k\) and the definition of \(P_k\) in (\ref{eqn:P_k}), we have
		\begin{eqnarray*}
			\sum\limits_{k=0}^\infty P_k(x) & = & \sum\limits_{k=0}^\infty \left( \sum\limits_{0 \leq \alpha \leq \N } m_\alpha^k G_\alpha^k(x)\right) = \sum\limits_{0 \leq \alpha \leq \N} \sum\limits_{k=0}^\infty \left(m_\alpha^k \lvert E_k \rvert \right) \left( \frac{1}{\lvert E_k \rvert}G_\alpha^k(x)\right)\\
			& = & \sum\limits_{0 \leq \alpha \leq \N} \sum\limits_{k=0}^\infty \left( \sum\limits_{j=k+1}^\infty m_\alpha^j \lvert E_j \rvert \right) \left( \frac{1}{\lvert E_{k+1}\rvert} G_\alpha^{k+1}(x) - \frac{1}{\lvert E_k \rvert} G_\alpha^k(x) \right)
		\end{eqnarray*}
		where the last equality holds due to \(\sum_{j=0}^\infty m_\alpha^j \lvert E_j \rvert = 0\).
		For \(0 \leq \alpha \leq \N\) and \(k\in\mathbb{N}_0\), let us define
		\begin{equation}
			N_\alpha^k := \sum\limits_{j=k+1}^\infty m_\alpha^j \lvert E_j \rvert \quad \text{ and }\quad h_\alpha^k(x) := \frac{1}{\lvert E_{k+1} \rvert} G_\alpha^{k+1}(x) - \frac{1}{\lvert E_k \rvert} G_\alpha^k(x),
			\label{eq: N and h}
		\end{equation}
 		and let \(\mu_\alpha^k\) satisfy 
		\begin{equation}
			0<\mu_\alpha^k \leq 3^{1/p} \Calphamax \left(1 + 2^{-\alpha-1} \right) (2^{k-1} \zeta \lvert I \rvert)^{1/p-\alpha-1}.		
		\label{eqn:claimresult2}
		\end{equation}
		Then, since \(h_\alpha^k/\mu_\alpha^k\) is a \(\left(p,2\right)\)-atom by Lemma~\ref{claim:P_k atom},
		\begin{equation*}
			\left\| \sum\limits_{k=0}^\infty P_k \right\|_{H^p}^p = \left\|\sum\limits_{0 \leq \alpha \leq \N } \sum\limits_{k=0}^\infty \left( N_\alpha^k \mu_\alpha^k \right) \,  \left(h_\alpha^k/\mu_\alpha^k\right) \right\|_{H^p}^p \leq  \sum\limits_{0 \leq \alpha \leq \N} \sum\limits_{k=0}^\infty \lvert N_\alpha^k \mu_\alpha^k\rvert^p.
		\end{equation*}
		Using the assumption (\ref{MBddCondition}), with \(\widetilde{E}_k:=\cup_{j=k+1}^\infty E_j\), we have
		\[
		\lvert N_{\alpha}^k \rvert = \Biggl\lvert \sum\limits_{j=k+1}^\infty \displaystyle\int M_j(x) x^\alpha dx \Biggr\rvert \leq  \int_{\widetilde{E}_k} \lvert M(x)\rvert \lvert x \rvert^\alpha dx 
		\leq C_M\lvert I \rvert^{b-1/p} \displaystyle\int_{\widetilde{E}_k} \lvert x \rvert^\alpha/ \lvert x \rvert^b dx.
		\]
		Since \(b>2/p\), \(0 \leq \alpha \leq \N=\left\lfloor 1/p -1 \right\rfloor\), we get \(b-\alpha-1>0\) and, by definition of \(E_j\), 	
		\begin{equation*}
			\lvert N_\alpha^k \rvert \leq C_M  \left( \frac{2}{b-\alpha-1} \right) (2^{k-1} \zeta)^{-b+\alpha+1} \lvert I \rvert^{\alpha+1-1/p}.
		\end{equation*}
		Using the bound for \(\mu_\alpha^k\), and the estimation \(\sum_{k=0}^\infty 2^{(1-bp)k}< 2\) again, this gives
		\begin{equation}
			\left\| \sum\limits_{k=0}^\infty P_k \right\|_{H^p}^p <  3 C_M^p 2^{bp} \zeta^{1-bp} \Calphamax^p \sum\limits_{0 \leq \alpha \leq \N} \left(\frac{2+ 2^{-\alpha}}{b-\alpha-1}  \right)^p.
			\label{HpNormofPk}
		\end{equation}
		Combining the bounds in (\ref{HpNormofMk-Pk}) and (\ref{HpNormofPk}) gives the stated bound for \(\left\| M\right\|_{H^p}^p\).
	\end{proof}

	The following proof of Theorem~\ref{thm:CZO has Hp bddness} corresponds to the final step, which combines the previous propositions to show that the Calder\'{o}n-Zygmund operator \(Z\) is bounded on \(H^p\). In the proof, we use an argument from \cite{[Bownik]AniHardy}.

	\begin{proof}[Proof of Theorem~\ref{thm:CZO has Hp bddness}]	
	Let \(Z\) be the Calder\'{o}n-Zygmund operator with a constant \(\gCZ\). 
	
	For each \((p,2)\)-atom \(\atom\) supported in an interval~\(I\) centered at \(y_0\), by Proposition~\ref{prop: CZO maps atoms to molecule},  we see that \(Z\atom\in L^2\) and 
	\begin{equation*}
		\lvert (Z\atom)(x)\rvert \leq \frac{C_3 \gCZ \lvert I \rvert^{b-1/p}}{\lvert x-y_0 \rvert^b},\quad \text{for } x\in\left(\zeta I\right)^c
	\end{equation*}
	where \(C_3=C_3(b,\zeta)\) is as in (\ref{eq:C_3}), and \(b>2/p\) and \(\zeta \geq \delta=\delta(b)\) are fixed.

Since \(\atom\) is an atom with sufficient vanishing moments, \(Z \atom\) satisfies the vanishing moment condition \(\int_{\mathbb{R}} Z \atom(x) \,x^\alpha dx = 0\) for all \(0 \leq \alpha \leq \N\) by the condition (\ref{CZCond3}) of the Calder\'{o}n-Zygmund operator. Hence, by Proposition~\ref{prop: Molecule is in Hp} with \(M:=Z\atom\), and \(C_M:= C_3 \gCZ\), we have \(Z\atom\in H^p\) with 
	\begin{equation}
		\left\| Z \atom \right\|_{H^p}^p \leq C_1 \zeta^{p(1/p-1/2)} \lVert Z \rVert_{L^2\to L^2}^p + C_3^p C_4 \gCZ^p,
\label{ZaIsBoundedNew}
	\end{equation}
	where the size condition of the atom \(\atom\) is used to replace \(\left\| Z \atom \right\|_{L^2}\) by \(\lVert Z \rVert_{L^2\to L^2}\). By cancelling out the terms \((\zeta/2)^{bp}\) in \(C_3^p\) and \((2/\zeta)^{bp}\) in \(C_4\), we see that \(C_2\) in the statement of Theorem~\ref{thm:CZO has Hp bddness} is in fact the same as \(C_3^p C_4\).
	
	To extend the bound of \(\left\| Z \atom \right\|_{H^p}\) in (\ref{ZaIsBoundedNew}) to that of \(\lVert Zf \rVert_{H^p}\), for any \(f\in H^p\), we consider the function space
	\begin{equation*}
		\Theta_{\N} := \left\{f\in L^2: f \text{ has a compact support, and } \int f(x)x^\alpha dx = 0, \forall 0 \leq \alpha \leq \N\right\}.
	\end{equation*}
	This is a well-known dense subspace of the Hardy space \(H^p\) \cite{[Bownik]AniHardy,[GarciaFrancia]WNI,[Stein]HA}. 	
		Let \(f\in\Theta_{\N}\). Since \(f\in H^p\), for any fixed \(\epsilon>0\), there exists an atomic decomposition such that \(f = \sum_{k=0}^{\infty} \lambda_k \atom_k\), where \(\atom_k\) are \((p,2)\)-atoms and \(\sum_{k=0}^{\infty}\lvert \lambda_k \rvert^p < (1+\epsilon) \left\| f \right\|_{H^p}^p< \infty\). 
		
		Let \(f_N = \sum_{k=0}^{N} \lambda_k \atom_k\), and let \(g_N:=Z(f_N)\). For \(N_1< N_2\), by using the bound in (\ref{ZaIsBoundedNew}) with \(C_3^p C_4\) replaced by \(C_2\), we have
		\[\left\| g_{N_2} - g_{N_1} \right\|_{H^p}^p = \left\| Z(f_{N_2} - f_{N_1}) \right\|_{H^p}^p \leq \sum\limits_{k=N_1+1}^{N_2} \lvert \lambda_k \rvert^p \left\| Z \atom_k \right\|_{H^p}^p\le \widetilde{C} \sum\limits_{k=N_1+1}^{N_2} \lvert \lambda_k \rvert^p,
		\]
		with \(\widetilde{C}:= C_1 \, \zeta^{p\left(1/p-1/2\right)} \left\|Z \right\|_{L^2\to L^2}^p + C_2 \gCZ^p\). Since \(\sum_{k=0}^{\infty}  \lvert \lambda_k \rvert^p < \infty\), \(g_N\) is a Cauchy sequence in \(H^p\), and thus converges to some element in \(H^p\). 
		
		Also, since \(Zf = \sum_{k=0}^{\infty} \lambda_k (Z \atom_k)\) converges in \(H^p\), for example, by the argument in Theorem~9.8 of \cite{[Bownik]AniHardy}, we conclude that \(g_N\) converges to \(Zf\) in \(H^p\). Thus, using (\ref{ZaIsBoundedNew}) again as above, we get
		\[
		\left\| Zf \right\|_{H^p}^p = \lim\limits_{N\to\infty} \left\| \sum\limits_{k=0}^N \lambda_k Z \atom_k \right\|_{H^p}^p \leq \widetilde{C} \sum\limits_{k=0}^\infty \lvert \lambda_k \rvert^p < \widetilde{C}(1+\epsilon) \left\| f \right\|_{H^p}^p.
		\]
		Since this holds for any arbitrary \(\epsilon>0\), 
		\(\left\| Zf \right\|_{H^p}^p \leq \widetilde{C} \left\| f \right\|_{H^p}^p\) for all  \(f\in\Theta_{\N}\). Furthermore, since \(\Theta_{\N}\) is dense in \(H^p\), we have the same bound for any \(f\in H^p\), i.e.,
		\begin{equation*}
			\left\| Zf \right\|_{H^p}^p \leq \widetilde{C} \left\| f \right\|_{H^p}^p, \quad \forall f \in H^p.
		\end{equation*}
		This completes the proof.
	\end{proof}

	\section{Boundedness and Invertibility of Wavelet Frame Operator on $H^p(\mathbb{R})$}
	\label{sec: Proof of Theorem2,3}
	
	This section consists of proofs of Theorem~\ref{Theorem:Bdd of U} and \ref{Theorem:Invertibility of U}. Recall that Theorem~\ref{Theorem:Bdd of U} gives sufficient conditions for the wavelet frame operator \(U\) to be a Calder\'{o}n-Zygmund operator, and have a bounded extension to \(H^p\). Sufficient conditions for the wavelet frame operator \(U\) to be invertible are given in Theorem~\ref{Theorem:Invertibility of U}.
	
	\begin{proof}[Proof of Theorem~\ref{Theorem:Bdd of U}]
		Let \(\psi,\phi\in L^2\) satisfy all the assumptions in Theorem~\ref{Theorem:Bdd of U}. Then the wavelet frame operator \(U = U_{\psi,\phi}\) is a Calder\'{o}n-Zygmund operator by Theorem~\ref{Result: L2boundedness},~\ref{thm:CZ kernel smoothness},~and \ref{thm:CZ kernel vanishing} with a constant 
		\begin{equation}
			\CZ(\psi,\phi):=\max\limits_{0 \leq \alpha \leq  \N+1} \CZ_{\alpha}(\psi,\phi)
			\label{eq:CZ_2}
		\end{equation}
		where \(\CZ_{\alpha}(\psi,\phi)\) is defined as in (\ref{eq:CZ_2,alpha}). Thus, by Theorem~\ref{thm:CZO has Hp bddness}, \(U\) is a bounded operator on \(H_p\) with the operator norm satisfying
			\[\left\|U \right\|_{H^p \to H^p}^p \leq C_1 \zeta^{p(1/p-1/2)} \left\| U \right\|_{L^2\to L^2}^p + C_2 \CZ(\psi,\phi)^p.\]
		Here, the values \(C_1\) and \(C_2 =C_2(b,\zeta)\) are the same as in Theorem~\ref{thm:CZO has Hp bddness}: 
		\begin{equation}
			C_1 = \left( 1+ \Calphamax \left( \N +1 \right)\right)^p,
			\label{eq:C_1}
		\end{equation}
		\begin{eqnarray}
			C_2(b,\zeta)& = & \left\{\frac{2^{\N+3}}{(\N+1)! \sqrt{2\N+3}}\right\}^p \left\{ \frac{1}{2} \left( \frac{1}{(\zeta + 1)^{2\N+3}} + \frac{1}{(\zeta -1)^{2\N +3}}\right) \zeta^{\frac{2}{p}-1}\right\}^{p/2} \nonumber \\
			&  &  \times \left\{ 2\cdot (2b-1)^{-p/2} C_1 + 3 \Calphamax^p \sum\limits_{0 \leq \alpha \leq \N }  \left(\frac{2 + 2^{-\alpha}}{b-\alpha-1} \right)^p\right\},
			\label{eq:C_2}
		\end{eqnarray}
		and they are finite for any fixed \(b>2/p\) and \(\zeta \geq \delta(b)\), with \(\delta(b)\) and \(\Calphamax\) given as in (\ref{eq:b and delta}) and (\ref{eqn : g}), respectively. Hence, the above upper bound for \(\left\|U \right\|_{H^p \to H^p}^p\) is finite for any fixed \(b\) and \(\zeta\), and this completes the proof.
	\end{proof}
	
	\begin{proof}[Proof of Theorem~\ref{Theorem:Invertibility of U}]
		Let \(\psi, \phi, \psi^*, \phi^* \in L^2\) satisfy all the assumptions of Theorem~\ref{Theorem:Invertibility of U}. Since \(\psi,\phi\) satisfy the decay condition (\ref{eq:BLdecay}), we see that \(U=s_\psi \circ t_\phi\) in \(L^2\) from Definition~\ref{def:synthesis,analysis operator}. Also, since \(\psi^*, \phi^* \in L^2\) are exact duals in \(L^2\), \(s_{\psi^*}\circ t_{\phi^*} = Id\) in \(L^2\). Then, by linearity of the operators, we have 
		\[ U- Id = s_\psi \circ t_\phi - s_{\psi^*} \circ t_{\phi^*} = s_{\psi-\psi^*} \circ t_{\phi} + s_{\psi^*} \circ t_{\phi-\phi^*} = U_{\psi-\psi^*,\phi} + U_{\psi^*,\phi-\phi^*}.\]
		Since \(\psi,\psi^*,\phi,\phi^* \) satisfy the decay condition (\ref{eq:BLdecay}) and the regularity condition, \(U-Id \) is a bounded operator on \(L^2\) by Theorem \ref{Result: L2boundedness}. Moreover, since they satisfy all the conditions in Theorem \ref{Theorem:Bdd of U}, the operator \(U-Id\) has a bounded extension to \(H^p\) with the following bound on the operator norm:
		\begin{eqnarray*}
			\left\| U-Id \right\|_{H^p \to H^p}^p & \leq & \left\| U_{\psi-\psi^*, \phi} \right\|_{H^p \to H^p}^p + \left\| U_{\psi^*,\phi-\phi^*} \right\|_{H^p \to H^p}^p \nonumber \\
			& \leq & C_1 \zeta^{p(1/p-1/2)} \left\| U_{\psi-\psi^*,\phi} \right\|_{L^2 \to L^2}^p + C_2(b,\zeta) \CZ(\psi-\psi^*,\phi)^p \nonumber \\ 
			& & + C_1 \eta^{p(1/p-1/2)} \left\| U_{\psi^*, \phi- \phi^*} \right\|_{L^2 \to L^2}^p + C_2(b,\eta) \CZ(\psi^*,\phi-\phi^*)^p,
		\end{eqnarray*}
		where \(\CZ(\cdot,\cdot)\) denotes the constant defined as in (\ref{eq:CZ_2}), and the constants \(C_1\) and \(C_2(b,\cdot)\) are the same as in the proof of Theorem~\ref{Theorem:Bdd of U} (c.f. (\ref{eq:C_1}) and (\ref{eq:C_2})) for \(b>2/p\) and \(\zeta, \eta \geq \delta(b)\), with \(\delta(b)\) given as in (\ref{eq:b and delta}).

		Now, let \(b>p/2\) be arbitrary but fixed. Then, \(\delta(b)\) is fixed as well and, by taking the infimum over all \(\zeta, \eta \geq \delta(b)\) in the above bound, we get
		\[\left\|U-Id\right\|_{H^p\to H^p}^p \leq \Mp,\]
		with \(\Mp\) defined as the resulting infimum value.
		The given assumption \(\Mp<1\) implies that \(\left\|U-Id\right\|_{H^p\to H^p}<1\), i.e., \(\psi\) and \(\phi\) are approximate duals in \(H^p\). Thus, the wavelet frame operator \(U\) is bijective on \(H^p\) by the Neumann series expansion \cite{NeumannArticle,NeumannBook}, and this completes the proof of Theorem~\ref{Theorem:Invertibility of U}. 
	\end{proof}

		\section{Wavelet Series Expansion in $H^p(\mathbb{R})$ }
		\label{sec:proof of corollary}
		
		In this section, we will show Theorem~\ref{Corollary} by using the previous results such as the boundedness and the invertibility of the wavelet frame operator \(U\), and two additional results. One of these additional results is Theorem~\ref{thm:Extension of S,T to TL}, and it says that each of the synthesis operator and the analysis operator has a bounded extension to the space related with \(H^p\). This is a generalization of the result for the dyadic case with \(A=2\) by Frazier and Jawerth \cite{[FJ]DD}, and this generalization is proved in \cite{[BownikHo]AnisotropicTLspace} for an arbitrary dilation factor \(A\in\mathbb{R}\), \(A>1\). The other additional result is Theorem~\ref{thm:HardyAadicTLEquiv}, and it gives the equivalence of the (non-dyadic) Triebel-Lizorkin space \(\dot{F}^{02}_p\) and the Hardy space \(H^p\) using a generalized Littlewood-Paley theory. One can show this result simply by adjusting the well-known assumptions for the dyadic case to the non-dyadic settings. 
		
		Below is the definition of the (non-dyadic) Triebel-Lizorkin space \(\dot{F}^{02}_p\) and the (non-dyadic) Triebel-Lizorkin sequence space \(\dot{f}^{02}_p\). Here, we use the notation \(\mathcal{S}'/\mathcal{P}\) for the space of tempered distributions modulo polynomials.

		\begin{definition}
			Let \(0<p \leq 1\).
			\begin{enumerate}[(a)]
				\item The \textit{(non-dyadic) Triebel-Lizorkin space} \(\dot{F}^{02}_p:=\dot{F}^{02}_p(\mathbb{R},A)\) is the collection of all \(f\in \mathcal{S}'/\mathcal{P}\) such that 
				\begin{equation*}
					\left\| f \right\|_{\dot{F}^{02}_p}:= \left\| \left( \sum\limits_{j \in\mathbb{Z}} \lvert f * \rho_{A,j} \rvert^2 \right)^{1/2} \right\|_{L^p} < \infty,
				\end{equation*}
				where \(\rho_{A,j}(\cdot)=A^j\rho(A^j\cdot)\) and \(\rho\) is a Schwartz function with its Fourier transform \(\widehat{\rho}\) supported in the annulus \(1/\sqrt{A} \leq \lvert \xi \rvert \leq A\), and \(\lvert \widehat{\rho}\rvert \geq C>0\) on the annulus \(1 \leq \lvert \xi \rvert \leq \sqrt{A}\).\label{def:TLspace}
				
				\item The \textit{(non-dyadic) Triebel-Lizorkin sequence space} \(\dot{f}^{02}_p:= \dot{f}^{02}_p(A)\) is the collection of all complex-valued sequences \(c=\{c_{jk}\}_{j,k\in\mathbb{Z}}\) such that 
				\begin{equation*}
					\left\| c \right\|_{\dot{f}^{02}_p}:= \left\| \left( \sum\limits_{j,k\in\mathbb{Z}} \left(\lvert c_{jk} \rvert A^{j/2} \chi_{I_{jk}}\right)^2\right)^{1/2} \right\|_{L^p}<\infty,
				\end{equation*}
				with the characteristic function \(\chi_{I_{jk}}\) on the interval 
		\(I_{jk} =[A^{-j} k, A^{-j} (k+1))\).\label{def:TLSequenceSpace}
			\end{enumerate}
			\label{def:TLspace,sequenceSpace}
		\end{definition}
		
		In \cite[Definition 5.1, Theorem 5.5 and 5.6]{[BownikHo]AnisotropicTLspace}, Bownik and Ho define functions such as the smooth synthesis molecule, say \(\psi_{BH}\), and the smooth analysis molecule, say \(\phi_{BH}\). With these functions, they show that each of the synthesis operator \(s_{\psi_{BH}}\) and the analysis operator \(t_{\phi_{BH}}\) has a bounded extension. We also remark that these results are initially proved in \cite{[FJ]DD} under the dyadic setting (i.e., \(A=2\)). 
		
		Under the assumptions given to our synthesizer \(\psi\) and analyzer \(\phi\) in Theorem~\ref{Theorem:Bdd of U}, if we add one more assumption (namely, \(\xi\hat{\psi} \in W^{\N+2,1}\)), we can easily show that these are precisely the smooth synthesis and analysis molecules of Bownik and Ho. Therefore, we obtain the following theorem by using the results in \cite{[BownikHo]AnisotropicTLspace}. We omit the proof of this theorem as it is straightforward.
		
		\begin{theorem}
			Let \(0<p \leq 1\) be fixed, and recall \(\Np= \lfloor 1/p-1 \rfloor\). Assume that \(\psi,\phi \in L^2\) satisfy all the conditions in Theorem \ref{Theorem:Bdd of U}, and further assume that \(\xi\widehat{\psi} \in W^{\Np+2,1}\). Then
			\begin{enumerate}[(a)]
				\item the wavelet analysis operator \(t_\phi:\dot{F}^{02}_p \to \dot{f}^{02}_p\) is a bounded operator, and
				\item the wavelet synthesis operator \(s_\psi:\dot{f}^{02}_p \to \dot{F}^{02}_p\) is a bounded operator.
			\end{enumerate}
			\label{thm:Extension of S,T to TL}
		\end{theorem}
		
		In the dyadic setting, it is well-known that the Hardy space and the dyadic Triebel-Lizorkin space are equivalent \cite{[Grafakos]MFA,[Stein]HA}. In the following theorem, we state that the Hardy space \(H^p\) in Definition~\ref{def:AtomicHardy} and the non-dyadic Triebel-Lizorkin space \(\dot{F}^{0 2}_p\), \(0<p \leq 1\), are equivalent. It can be shown by using the Littlewood-Paley theory under the non-dyadic setting. We omit the proof of this theorem as well because it results from simple modifications of the known techniques. 
		
		\begin{theorem}
			Let \(\rho\) be a Schwartz function whose Fourier transform is nonnegative, supported in the annulus \(1/\sqrt{A} \leq \lvert \xi \rvert \leq A\), equals to \(1\) on the smaller annulus \(1 \leq \lvert \xi \rvert \leq \sqrt{A}\), and satisfies 
			\[\sum\limits_{j\in\mathbb{Z}} \widehat{\rho}(A^{-j} \xi ) =1, \quad \forall \xi \neq 0.\]
			Let \(0<p \leq 1\) be fixed. Then, the following statements hold true.
		\begin{enumerate}[(a)]
		\item 
For all \(f\in H^p\), we have \(\left\| f \right\|_{\dot{F}^{02}_p} \lesssim \left\| f \right\|_{H^p}\).
			\item 
			Conversely, if a tempered distribution \(f\) satisfies \(\left\| f \right\|_{\dot{F}^{02}_p}<\infty\),
			then there exists a unique polynomial \(g\) such that \(f-g\in H^p\) and satisfies \(\left\| f-g \right\|_{H^p} \lesssim \left\| f \right\|_{\dot{F}^{02}_p}.\)
	    \end{enumerate}
			\label{thm:HardyAadicTLEquiv}
		\end{theorem}
		
		By using Theorem~\ref{thm:Extension of S,T to TL}  and  \ref{thm:HardyAadicTLEquiv}, let us show Theorem~\ref{Corollary}. 
		
		\begin{proof}[Proof of Theorem~\ref{Corollary}]
			Let \(\psi, \phi\in L^2\) be functions satifying all the conditions in the theorem. Then the wavelet frame operator \(U=U_{\psi,\phi}\) is a bounded operator from the Hardy space \(H^p\) to itself by Theorem~\ref{Theorem:Bdd of U}, and the operator is invertible by Theorem \ref{Theorem:Invertibility of U}. Let \(f\in H^p\) be fixed. Then, \(U^{-1}f\) is in \(H^p\) and thus in \(\dot{F}^{02}_p\) by Theorem~\ref{thm:HardyAadicTLEquiv} (a). By Theorem \ref{thm:Extension of S,T to TL} (a), we have 
			\(t_{\phi} (U^{-1} f) = \left\{ \left< U^{-1} f, \phi_{jk} \right>\right\}_{j,k\in\mathbb{Z}} \in \dot{f}^{02}_p\). 
			Subsequently, by applying the synthesis operator \(s_\psi\) to the sequence \(t_{\phi}(U^{-1}f)\) and by using Theorem~\ref{thm:Extension of S,T to TL}~(b), we have
			\(s_{\psi}(t_{\phi}(U^{-1} f))\in \dot{F}^{02}_p\).
			Then, by Theorem \ref{thm:HardyAadicTLEquiv} (b), there is a unique polynomial \(g\) such that 
			\[s_\psi(t_{\phi}(U^{-1} f)) - g \in H^p.\]
			However, by definition of wavelet frame operator \(U\) and from the fact that  \(U^{-1} f \in H^p\), we see that \(U(U^{-1}f) = s_{\psi}(t_{\phi}(U^{-1} f))\). Since \(U\) is bounded from \(H^p\) to \(H^p\), we have \(U(U^{-1}f) \in H^p\), which in turn implies \(s_{\psi}(t_{\phi}(U^{-1} f) )\in H^p\). Therefore, we conclude that  the unique polynomial \(g\) must be \(0\) and that
			\[f = \sum\limits_{j,k\in\mathbb{Z}} \left<U^{-1} f,\phi_{jk} \right> \psi_{jk}.\]
			In particular, for any \(f\in H^p\), there exists \(\{c_{jk}\} = \left\{ \left< U^{-1}f, \phi_{jk} \right> \right\}_{j,k\in\mathbb{Z}} \in \dot{f}^{02}_p\) such that
			\(f = \sum\limits_{j,k\in\mathbb{Z}} c_{jk} \psi_{jk}\).
		\end{proof}
	
	\section{Example for $H^p$ with $1/2<p \leq 1$}
	\label{sec:Example}
	
	In this section we show that the Mexican hat function used for the space $H^1$ in \cite{[BuiLaug]WFBonLH} can also be used for the space $H^p$  with \(1/2< p \leq 1\). It will serve as an example illustrating our main results. 
	
	Let \(\psi(x) = (1-x^2) e^{-x^2/2}\) be the Mexican hat function, which will be our synthesizer. Let \(A=2\) and \(1/2< p \leq 1\) be fixed. Then \(\Np=\lfloor 1/p - 1 \rfloor =0\).
	
	We follow the setting of \cite[Chapter~6]{[BuiLaug]WFBonLH} in choosing the corresponding analyzer \(\phi\), and the exact duals \(\psi^*\) and \(\phi^*\) required in Theorem~\ref{Theorem:Invertibility of U}. In particular, the two functions \(\phi\) and \(\phi^*\) are chosen to be the same.
	
	Then, since \(\psi,\psi^*,\phi,\phi^* \) satisfy all the assumptions in Theorem~\ref{Theorem:Invertibility of U}, and since \(\phi^* = \phi\), the wavelet frame operator \(U_{\psi,\phi}\) satisfies 
	\begin{equation}
		\left\| U_{\psi,\phi}-Id \right\|_{H^p\to H^p}^p \leq \left\| U_{\psi-\psi^*,\phi} \right\|_{H^p \to H^p}^p \leq C_1 \zeta^{p(1/p-1/2)} \mathcal{U}^p + C_2 \CZ^p,
		\label{eqn : example_bound of Hp operator norm}
	\end{equation}
	with \(\zeta \geq \delta\) where \(\delta = 3/(4b) + \sqrt{1+ 9/({16b^2})}\) for \(b>2/p\), as seen from its proof, \(\mathcal{U}:=\lVert U_{\psi-\psi^*, \phi} \rVert_{L^2 \to L^2}\) and \(\CZ:= \CZ(\psi-\psi^*,\phi)\). Here, \(C_1=\left( 1+ \Calphamax_0 \right)^p\),
	\[ 
	C_2 = C_2(b,\zeta) = \left(\frac{8}{\sqrt{3}}\frac{(\zeta^2+3)^{1/2}}{(\zeta^2-1)^{3/2}}\right)^{p} \left(2 \cdot \left( 2b-1\right)^{-p/2} C_1 + 3 \cdot \left(\frac{3\Calphamax_0}{b-1}\right)^p \right) \zeta,
	\]
	where \(\Calphamax_0\) is a constant satisfying \eqref{eq:Calpha} with \(\alpha=0\).
	
	Since, for each \(k \in \mathbb{N}_0\), the constant function \(g_0^k(x)=1\) and the sets $E_k$ defined in \eqref{eq:E_k} satisfy 
		\[\frac{1}{\lvert E_k \rvert} \displaystyle \int_{E_k} g_0^k(x) dx = 1,\]
	we can take \(\Calphamax_0=1\). For this choice of \(\Calphamax_0\), \(C_1=2^p\). 

	Now using \(C_1=2^p\) and \(\Calphamax_0=1\), we see that the right-hand side of (\ref{eqn : example_bound of Hp operator norm}) is bounded by
	\begin{equation} 
		2^{1+1/p} \zeta^{1/p}\left(\frac{\mathcal{U}}{\zeta^{1/2}} + \frac{4}{\sqrt{3}}\frac{(\zeta^2+3)^{1/2}}{(\zeta^2-1)^{3/2}} \left( \frac{2\cdot 4^{1/p}}{(2b-1)^{1/2}}  + \frac{3\cdot 6^{1/p}}{b-1} \right) \CZ \right).
		\label{eq:Np bound 1}
	\end{equation}
	The estimation \(\mathcal{U}< 0.00026\) is given in \cite[p.405]{[BuiLaug]WFBonLH}. Since \(\Np =0\) and \(A=2\), from \eqref{eq:CZ_2} and \eqref{eq:CZ_2,alpha}, we see that 
	\[ \CZ = \CZ(\psi-\psi^*,\phi)= \max \left\{4 (\sigma_0\tau_0)^{1/2}, \frac{10}{3} (\sigma_1\tau_1^2)^{1/3}\right\},\]
	where \(\sigma_i:=\sigma_i(\psi-\psi^*,\phi)\) and \(\tau_i:=\tau_i(\psi-\psi^*,\phi)\) for \(i=0,1\). The estimation for \(\sigma_i\) and \(\tau_i\), \(i=0,1\), can be found in \cite[p.403-404]{[BuiLaug]WFBonLH}, and we have
	\[\sigma_0<0.000045, \quad \sigma_1<0.00022, \quad \tau_0<0.00086, \quad \tau_1<0.036,\]
	which results in \(\CZ<0.022\).
	
	Hence, for any \(\zeta\ge \delta = 3/(4b) + \sqrt{1+ 9/({16b^2})}\) with \(b>2/p\), the number in \eqref{eq:Np bound 1} is bounded by
	\[	2^{1+1/p} \zeta^{1/p}\left(\frac{0.00026}{\zeta^{1/2}} + \frac{4\cdot 0.022}{\sqrt{3}}\frac{(\zeta^2+3)^{1/2}}{(\zeta^2-1)^{3/2}} \left( \frac{2\cdot 4^{1/p}}{(2b-1)^{1/2}}  + \frac{3\cdot 6^{1/p}}{b-1} \right) \right).\]
	Let \(\zeta=5\) and \(b=250\). Then the required conditions \(\zeta \geq \delta\) and \(b>2/p\) are satisfied for all \(1/2<p \leq 1\). Once we substitute these values to the above estimation, the only variable in the resulting expression is \(p\) and this expression can be bounded by the value with \(p=1/2\). Therefore, we see that for all \(1/2<p \leq 1\), \(\left\| U_{\psi,\phi}-Id \right\|_{H^p\to H^p}\) is bounded by 
		\[ 2^3 5^2\left(\frac{0.00026}{5^{1/2}} + \frac{4\cdot 0.022}{\sqrt{3}}\frac{(5^2+3)^{1/2}}{(5^2-1)^{3/2}} \left( \frac{2\cdot 4^{2}}{(2\cdot 250-1)^{1/2}}  + \frac{3\cdot 6^{2}}{250-1} \right) \right)< 1.\]
		
	We conclude that \(U_{\psi,\phi}\) is bijective on \(H^p\), and since the Mexican hat function satisfies the additional regularity condition \(\xi \widehat{\psi} \in W^{2,1}\) in Theorem~\ref{Corollary} as well, every element in Hardy space \(H^p\) has a wavelet series expansion with a wavelet system \(\{\psi_{jk}\}\) using the Mexican hat function \(\psi\) for all \(1/2<p \leq 1\).

	\begin{appendices}
		\label{Appendix}
		\section{Lemmas for Section~\ref{sec:WFO is CZO}}
		\setcounter{lemma}{0}
		\renewcommand{\thelemma}{\Alph{section}.\arabic{lemma}}
		
		\label{appx : lemmas needed in Sec2}
		
		In this section, we place some lemmas needed in Section~\ref{sec:WFO is CZO}. Every result is straightforward from each mentioned reference, so we omit the proof. 
		
		\medskip
		
		The following lemma is an immediate extension of Lemma 4.3 in \cite{[BuiLaug]WFBonLH} for the case \(l\ge 2\).
		
		\begin{lemma}
			Let \(l\in \mathbb{N}\). If \(\widehat{\psi}\in W^{l,1}\), then \(\psi\) is continuous and \(\lvert \psi(x) \rvert\lesssim {1/(1+ \lvert x \rvert^l)}\).
			\label{lemma:W^{n,1}}
		\end{lemma}
	
		The following lemma shows sufficient conditions to make the \(\alpha\)-th partial derivative \(\partial_y^\alpha K_0(x,y)\) of \(K_0(x,y)\) absolutely convergent and continuous for any fixed \(\alpha\in \mathbb{N}_{0}\). The restricted cases, \(\alpha=0\) and \(1\), of this lemma are proved in \cite{[BuiLaug]WFBonLH}. 
		
		\begin{lemma}
			Let \(\alpha \in \mathbb{N}_0\). Assume that \(\psi, \phi\) satisfy the following conditions:
			\begin{enumerate}[(a)]
				\item \(\phi\) and all derivatives of \(\phi\) up to the \(\alpha\)-th order (that is, \(\phi^{(0)},\cdots, \phi^{(\alpha)}\)) exist and are bounded,
				\item \( \lvert \psi(x) \rvert \lesssim 1/(1+\lvert x \rvert ^2)\).
			\end{enumerate} 
			Then \(\partial_y^\alpha K_0(x,y) = \sum_{k \in\mathbb{Z}} \psi(x-k) \overline{\phi^{(\alpha)} (y-k)}\) converges absolutely for each \(x,y\in\mathbb{R}\).
			Moreover, if we further assume the continuity of \(\psi\) and \(\phi^{(\alpha)} \), then \(\partial_y^\alpha K_0(x,y)\) is continuous.
			\label{lemma:AbsolutelyConvergenceContinuityK0}
		\end{lemma}

		The next lemma is a generalization of Lemma \(A.1\) in \cite{[BuiLaug]WFBonLH}. There, it is proved for \(l=0,1\) and we extend the result to the case \(l \ge 2\). We write our lemma for all \(l \in\mathbb{N}_0\).
		\begin{lemma}
			Let \(l\in \mathbb{N}_0\), \(A>1\) 
			and \(\sigma, \tau \geq 0\). If \(g(z)\leq \min \{\sigma, \tau/ \lvert z\rvert ^{l+2}\}\) for all \(z\neq 0\), then
				\[
				\sum\limits_{j\in\mathbb{Z}} \vert A^{j(l+1)} g(A^jz) \rvert \leq \kappa_l(A) \frac{\sigma^{1/(l+2)} \tau^{(l+1)/(l+2)}}{\lvert z \rvert^{l+1}}
				\]
				where \(\kappa_l(A) = A(2A^l + \sum_{k=0}^{l-1} A^k)/ (A^{l+1}-1)\), with \(\sum_{k=0}^{-1} A^k\) interpreted as 0.
			\label{lemma:appendix1}
		\end{lemma}

		The next lemma shows sufficient conditions to make the \(\alpha\)-th partial derivative \(\partial_y^\alpha K(x,y)\) of \(K(x,y)\) absolutely convergent and continuous, where we set \(K(x,y) := K_{\psi,\phi}(x,y)\). For the restricted cases of \(\alpha=0\) and \(1\), this result is proved in \cite{[BuiLaug]WFBonLH}. Here, we state the lemma for all cases \(\alpha \in\mathbb{N}_0\).
		
		\begin{lemma}
			Let \(\alpha\in\mathbb{N}_0\) and \(A>1\) be a dilation factor. 
			Assume that \(\psi,\phi\) satisfy the following conditions:
			\begin{enumerate}[(a)]
				\item \(\phi\) and all derivatives of \(\phi\) up to the \(\alpha\)-th order (that is, \(\phi^{(0)},\cdots, \phi^{(\alpha)}\)) exist and are bounded,
				\item \(\lvert \psi(x) \rvert, \lvert \phi^{(\alpha)}(x) \rvert \lesssim 1/(1+\lvert x \rvert^{\alpha+2}).\)
			\end{enumerate}
			Then
			\[\partial_y^\alpha K(x,y) = \sum\limits_{j \in\mathbb{Z}} \sum\limits_{k\in\mathbb{Z}} A^{j(\alpha+1)} \psi(A^jx - k) \overline{\phi^{(\alpha)}(A^j y - k) }\]
			converges absolutely and uniformly on \(\{(x,y): x \neq y\}\) with  \(\lvert \partial_y^\alpha K(x,y) \rvert \lesssim 1/\lvert x-y \rvert^{\alpha+1}\). If we further assume the continuity of \(\psi\) and \(\phi^{(\alpha)}\), then \(\partial_y^\alpha K\) is continuous on \(\{(x,y): x\neq y\}\). 
			\label{lemma:AbsoluteConvergenceEstimationK}
		\end{lemma}
	
		The next lemma is a generalization (c.f. \cite{Hardy-RellichIneq,Hardy-Ineq}) of the well-known Hardy's inequality in \cite{[Hardy]Hardy'sInequality}.
		
		\begin{lemma}[\cite{Hardy-RellichIneq,Hardy-Ineq}]
			If \(f\) is differentiable, \(f(0)=0\), and \(\int \lvert f'(x) \rvert^2/\lvert x \rvert^{k} dx< \infty\) for some \(k\ge 0\), then \(\int \lvert f(x) \rvert^2/\lvert x\rvert^{k+2} dx \leq (4/(k+1)^2) \int \lvert f'(x)\rvert^2/{\lvert x \rvert^k} dx\).
			\label{lemma:Hardy-R Inequality}
		\end{lemma}
		
		We can simply apply the above result repeatedly for higher-order derivatives to obtain the following lemma. 
		
		\begin{lemma}
			For any fixed \(l\in \mathbb{N}\), assume that
			\begin{enumerate}[(a)]
				\item \(f^{(l)}\) exists and \(\displaystyle f^{(l)}\) belongs to \(L^2\), and
				\item \(f(0)= f'(0)=\cdots=f^{(l-1)}(0) = 0\).
			\end{enumerate}
			Then we have \(\left\| f/x^l\right\|_{L^2} \lesssim \left\| f^{(l)} \right\|_{L^2}.\)
			\label{lemma:Generalized Hardy's Inequality}
		\end{lemma}
		
		\section{Example for Lemma~\ref{lemma : polynomial G}}
		\label{appx : example of Lemma9}		
		
		We depict a simple and explicit example of Lemma~\ref{lemma : polynomial G} in case of \(N=1\) and fixed \(R>0\). Specifically, for each fixed \(0\leq \alpha \leq 1\), we can easily find the unique polynomial~\(g_\alpha^k(x)\) for every \(k \in \mathbb{N}_0\), and a constant \(\Calpha\). Afterwards, we can also find the constant \(\Calphamax\) defined as in (\ref{eqn : g}) by taking maximum value among~\(\{\Calpha\}_{0 \leq \alpha \leq 1}\). 
		\begin{enumerate}[(a)]
			\item Let \(\alpha = 0\) and let \(g_0^k(x) = r_1 x + r_2\) satisfy
			\begin{equation*}
				\int_{E_k} (r_1 x + r_2) dx = \lvert E_k \rvert, \quad \int_{E_k} (r_1 x^2 + r_2 x) dx = 0.
			\end{equation*}
			Since \(E_k\) is symmetric, integral of odd-order term on \(E_k\) vanishes in the above equations. Then we have \(r_2=1\) from the first equation and \(r_1=0\) from the second one regardless of \(k\). That is, \(g_0^k\) is uniquely determined by~\(1\) for every \(k \in \mathbb{N}_0\). Afterwards, we can define \(G_0^k (x) = \chi_{E_k}(x)\) for every \(k\in \mathbb{N}_0\), and find \(\Calphamax_0 = 1\) such that \(\lvert G_0^k (x) \rvert \leq \Calphamax_0 \) holds for every \(x \in \mathbb{R}\).
			
			\item Similarly for \(\alpha =1\), by setting \(g_1^k(x) = r_1 x +r_2\) with
			\begin{equation*}
				\int_{E_k} (r_1 x + r_2) dx = 0, \quad \int_{E_k} (r_1 x^2 + r_2 x) dx = \lvert E_k \rvert,
			\end{equation*}
			we can get \(g_1^0(x) = (3/R^2) x\), and \(g_1^k(x) = (12/(7(2^kR)^2)) x\) for \(k\in \mathbb{N}\). Then we can define \(G_1^k(x) = g_1^k(x) \chi_{E_k}\) with each polynomial \(g_1^k(x)\) for~\(k\in \mathbb{N}_0\), and find \(\Calphamax_1 = 3\) such that \(\lvert G_1^k(x) \rvert \leq \Calphamax_1/(2^k R)\) holds. Especially note that \(\Calphamax_1\) is a constant independent of \(k\).
		\end{enumerate}
		From the above, we can determine \(\Calphamax = \max \{\Calphamax_0, \Calphamax_1\} = 3\).
		
		\section{Lemmas for Proposition~\ref{prop: Molecule is in Hp}}
		\setcounter{lemma}{0}
		\renewcommand{\thelemma}{\Alph{section}.\arabic{lemma}}
		\label{appx: Necessities for Thm6}
		
		In this section, we use the same notation and work under the same settings as in the proof of Proposition~\ref{prop: Molecule is in Hp}. The following lemma is used to prove Lemma~\ref{claim:M_k-P_k atom}. 
		
		\begin{lemma} 
			For any fixed \(k\in \mathbb{N}_0\), let \(M_k\) be the function defined as in (\ref{eqn:M_k}) and \(P_k\) be the polynomial defined in (\ref{eqn:P_k}). Then we have
			\[\left\|M_k-P_k\right\|_{L^2} \leq \left(1+\Calphamax \left( \N +1 \right)\right) \left\|M_k\right\|_{L^2}\]
			where \(\Calphamax\) and \(\N\) are constants in Proposition~\ref{prop: Molecule is in Hp}.	
			\label{lemma:M_k-P_k<M_k}
		\end{lemma}
		
		\begin{proof}
			Recall that for each \(0 \leq \alpha \leq \N\), there is a constant \(\Calpha\) such that 
			\(\lvert G_\alpha^k(x)\rvert \leq \Calpha (2^{k-1}\zeta \lvert I \rvert)^{-\alpha}\) for every \(x\in\mathbb{R}\) from Lemma~\ref{lemma : polynomial G} with \(N= \N\) and \(R=\zeta \lvert I \rvert/2\). By the definition of \(m_\alpha^k\) and the estimation of \(\lvert G_\alpha^k(x)\rvert\), we have
			\[
			\lvert P_k(x) \rvert \leq \sum\limits_{0 \leq \alpha \leq \N} \lvert m_\alpha^k\rvert  \lvert G_\alpha^k (x) \rvert \leq \sum\limits_{0 \leq \alpha \leq \N} \frac{1}{\lvert E_k \rvert} \biggl\lvert \displaystyle\int M_k(x) \,x^\alpha dx \biggr\rvert \left( \Calpha (2^{k-1} \zeta \lvert I \rvert)^{-\alpha} \right).
			\]
			Since the support of \(M_k\) is \(E_k\), by the definition of \(E_k\), we have \(\lvert x \rvert^\alpha \leq (2^{k-1} \zeta \lvert I \rvert)^\alpha\) for any \(x\in E_k\). Therefore, the integral in the right-hand side of the above inequality satisfies
			\(
			\lvert \int M_k(x) x^\alpha dx \rvert \leq (2^{k-1} \zeta \lvert I \rvert)^\alpha \int_{E_k} \lvert M_k(x) \rvert dx.
			\)
			Then, with \(\Calphamax =\max_{0 \leq \alpha \leq \N} \Calpha\), the absolute value of \(P_k\) satisfies 
			\begin{equation*}
				\lvert P_k(x) \rvert \leq \Calphamax \sum\limits_{0 \leq \alpha \leq \N} \frac{1}{\lvert E_k\rvert} \displaystyle\int_{E_k} \lvert M_k(x) \rvert dx \leq  \frac{\Calphamax \left( \N +1 \right)}{\lvert E_k \rvert^{1/2}} \left\|M_k\right\|_{L^2}
			\end{equation*}
			where the last inequality holds due to the H\"{o}lder's inequality. Since the support of \(P_k\) is \(E_k\), we have \(
			\left\|P_k\right\|_{L^2} \leq \Calphamax \left(\N +1 \right) \left\|M_k\right\|_{L^2}.
			\) Thus, the stated bound is obtained as
			\[\left\|M_k-P_k \right\|_{L^2} \leq \left\|M_k\right\|_{L^2} + \left\|P_k \right\|_{L^2} \leq \left(1+\Calphamax \left(\N + 1 \right)\right) \left\|M_k \right\|_{L^2},\]
			and this completes the proof.
		\end{proof}

		\begin{lemma}
			For any fixed \(k \in \mathbb{N}_0\), let \(M_k\) be a function defined as in (\ref{eqn:M_k}) and \(P_k\) be a polynomial defined in (\ref{eqn:P_k}). If \(\lambda_k> 0\) satisfies (\ref{eqn:claimresult1}), then \((M_k-P_k)/\lambda_k\) is a \(\left(p,2\right)\)-atom. 
			
			\label{claim:M_k-P_k atom}
		\end{lemma}	
		
		\begin{proof}
			Since the support of \(M_k\) and \(P_k\) is in \(E_k\), the support of \(M_k-P_k\) is contained in \(E_k\). Fix \(0 \leq \beta \leq \N\). By the definition of \(P_k\), we have
			\begin{equation*}
				\int \left(M_k - P_k\right) x^\beta dx = \int_{E_k} M(x)\,x^\beta dx - \sum\limits_{0 \leq \alpha \leq \N} m_\alpha^k \int_{E_k} G_\alpha^k (x) \,x^\beta dx.
			\end{equation*}
			The definition of \(m_\beta^k\), together with (\ref{phiConditionEk}), gives
			\[\sum\limits_{0 \leq \alpha \leq \N} m_\alpha^k \displaystyle\int_{E_k} G_\alpha^k (x) \,x^\beta dx = m_\beta^k \lvert E_k \rvert = \int M_k(x) x^\beta dx.
			\]
			Thus, we see that 
			\(\int \left(M_k - P_k\right) x^\beta dx = 0\)
			for every \(0 \leq \beta \leq \N \). This proves the vanishing moment condition for \(M_k - P_k\).
			Finally, we aim to find the range of \(\lambda_k\) satisfying
			\[\left\|(M_k-P_k)/\lambda_k\right\|_{L^2} \leq \lvert E_k\rvert^{1/2 - 1/p}. \]
			Let \(\lambda_k=\left\| M_k - P_k \right\|_{L^2} \lvert E_k \rvert^{-\left( 1/2-1/p\right)}\).
			For \(k=0\), by using Lemma~\ref{lemma:M_k-P_k<M_k} and the fact that \(\left\|M_0\right\|_{L^2} \leq \left\| M \right\|_{L^2}\), we have
			\begin{equation*}
				\lambda_0 \leq \left(1+\Calphamax \left( \N + 1 \right)\right) \left\|M \right\|_{L^2} \lvert E_0 \rvert^{-\left(1/2-1/p\right)} = \left(1+\Calphamax \left( \N + 1 \right)\right)(\zeta \lvert I \rvert)^{1/p-1/2} \left\|M \right\|_{L^2}
			\end{equation*}
			where the last equality holds since \(\lvert E_0 \rvert=\zeta \lvert I \rvert\).
			
			Since \(M\) satisfies the condition (\ref{MBddCondition}) (recall that \(y_0=0\)), for each \(k \in\mathbb{N}\), \(\left\| M_k \right\|_{L^2}\) is bounded by 
			\begin{equation*}
				C_M \lvert I \rvert^{b- 1/p} \left( \int_{E_k} \frac{1}{\lvert x \rvert^{2b}} dx \right)^{1/2} = C_M (2^{k-1} \zeta)^{1/2-b} \lvert I \rvert^{1/2-1/p}  \left(  \frac{2^{2b}-2}{2b-1}\right)^{1/2}.
			\end{equation*}
			With Lemma~\ref{lemma:M_k-P_k<M_k} and the above estimation of \(\left\|M_k \right\|_{L^2}\), for each \(k\in\mathbb{N}\), we have 
			\begin{equation*}
				\left\| M_k-P_k \right\|_{L^2} \leq C_M (2^{k-1} \zeta)^{1/2-b} \lvert I \rvert^{1/2-1/p} \left(\frac{2^{2b}-2}{2b-1}\right)^{1/2} \left(1+\Calphamax \left( \N +1 \right)\right).
			\end{equation*}
			Thus, since \(\lvert E_k \rvert \leq 2^{k-1} \zeta \lvert I \rvert\), we conclude that 
			\begin{equation*}
				\lambda_k \leq C_M (2^{k-1} \zeta)^{1/p-b} \left(\frac{2^{2b}-2}{2b-1}\right)^{1/2} \left(1+\Calphamax \left( \N +1\right)\right)
			\end{equation*}
		\end{proof}
		
		\begin{lemma}
			For any fixed \(0 \leq \alpha \leq \N\) and \(k\in\mathbb{N}_0\), let \(h_{\alpha}^k\) be a function defined as in~(\ref{eq: N and h}). If \(\mu_{\alpha}^k\) satisfies (\ref{eqn:claimresult2}), then \(h_\alpha^k/\mu_\alpha^k\) is a \(\left(p,2\right)\)-atom.
			\label{claim:P_k atom}
		\end{lemma}

		\begin{proof}
			Recall that \(h_\alpha^k(x) =(1/\lvert E_{k+1} \rvert) G_\alpha^{k+1}(x) - (1/\lvert E_k\rvert) G_\alpha^k(x)\). Since the support of  \(G_\alpha^{k+1}\) is in \(E_{k+1}\) and the support of \(G_\alpha^k\) is in \(E_k\), the support of \(h_\alpha^k\) is contained in \(E_k \cup E_{k+1}\). For any fixed  \(0 \leq \beta \leq \N\), we have
			\[\int h_\alpha^k(x) x^\beta dx = \frac{1}{\lvert E_{k+1} \rvert} \int G_\alpha^{k+1}(x) x^\beta dx - \frac{1}{\lvert E_k\rvert} \int G_\alpha^k (x) x^\beta dx = 0
			\]
			where the last equality holds due to (\ref{phiConditionEk}). This shows the vanishing moment condition for \(h_\alpha^k\). Next, we will find the range of \(\mu_\alpha^k\) satisfying
			\[\left\|h_\alpha^k/\mu_\alpha^k\right\|_{L^2} \leq \lvert E_k \cup E_{k+1}\rvert^{1/2-1/p}. \]
			Let \(\mu_\alpha^k =\left\|h_\alpha^k \right\|_{L^2} \lvert E_k \cup E_{k+1}\rvert^{-\left(1/2-1/p\right)} \). 
			Then since there is a constant \(\Calphamax\) such that 
			\(\lvert G_\alpha^k(x) \rvert \leq \Calphamax (2^{k-1} \zeta \lvert I \rvert)^{-\alpha}\) for every \(k\in \mathbb{N}_0\) and \(0 \leq \alpha \leq \N\), we have
			\begin{eqnarray*}
				\lvert h_\alpha^k(x)\rvert \leq  \frac{1}{\lvert E_{k+1} \rvert} \lvert G_\alpha^{k+1}(x)\rvert + \frac{1}{\lvert E_k \rvert} \lvert G_\alpha^k(x) \rvert \leq  \Calphamax \left(1 + 2^{-\alpha-1}\right) (2^{k-1} \zeta \lvert I \rvert)^{-\alpha - 1}.
			\end{eqnarray*}
			Using the fact that the support of \(h_\alpha^k\) is contained in \(E_k \cup E_{k+1}\), we obtain \(\left\|h_\alpha^k \right\|_{L^2} \leq \Calphamax \left( 1 + 2^{-\alpha-1} \right) (2^{k-1} \zeta \lvert I \rvert)^{-\alpha-1} \lvert E_k \cup E_{k+1} \rvert^{1/2}\) and conclude that
			\(
			\mu_\alpha^k \leq 3^{1/p} \Calphamax \left(1 + 2^{-\alpha-1}\right) (2^{k-1} \zeta \lvert I \rvert)^{1/p-\alpha-1}. 
			\) 
		\end{proof}
		
	\end{appendices}
	
	\bibliography{Hardy_revised}
	
	
\end{document}